\documentclass{article} 
\usepackage{fullpage}
\usepackage{amsmath,amsthm,amssymb}
\usepackage{todonotes}
\usepackage[utf8]{inputenc}

\usepackage{color, colortbl}
\definecolor{Gray}{gray}{0.9}

\newcommand{\nth}{^{\operatorname{th}}}

\newcommand{\mbf}[1]{{\bf #1}}

\title{A geometric lemma for complex polynomial curves with applications in Fourier restriction theory}
\author{Jaume de Dios Pont}

\newtheorem{thm}{Theorem}
\newtheorem{prop}{Proposition}
\newtheorem{lemma}{Lemma}
\newtheorem{defi}{Definition}
\newtheorem{example}{Example}
\newtheorem*{remark}{Remark}

\begin{document}
	\maketitle
	\abstract{
	The aim of this paper is to prove a uniform Fourier restriction estimate for certain $2-$dimensional surfaces in $\mathbb R^{2n}$. These surfaces are the image of complex polynomial curves $\gamma(z) = (p_1(z), \dots, p_n(z))$, equipped with the complex equivalent to the affine arclength measure. This result is a complex-polynomial counterpart to a previous result by Stovall \cite{stovall_uniform_2016} in the real setting.
	
	As a means to prove this theorem we provide an alternative proof of a geometric inequality by Dendrinos and Wright \cite{dendrinos_fourier_2010} that extends the result to complex polynomials.
	}

	\tableofcontents

	\newpage
	\section{Introduction} 
	\label{sec:introduction}

	Multiple results in harmonic analysis involving integrals of functions over manifolds (such as restriction theorems , convolution estimates, maximal function estimates) depend strongly on the non-vanishing of some curvature of the manifold. One of the most remarkable examples is the $L^p$-optimal restriction theorem for smooth curves with the arclength measure. Originally proven by Drury \cite{drury_restrictions_1985}, the result is only true for curves with non-vanishing torsion.
	
	There has been considerable interest \cite{pee_fourier_1974,drury1990degenerate,bak_restriction_2012,bak_restriction_2008,bruce_fourier_2019,dendrinos_fourier_2010,drury_fourier_1987,drury_fourier_1985} in proving analogous results in the degenerate case where the torsion vanishes at a finite number of points, by using the affine arc-length as an alternative measure. As a model case, multiple results have been proven in which the implicit constants do not depend on the polynomial curve, but only on its degree and the dimension of the ambient space. In the particular case of the restriction problem for curves, this culminated in 2016 with Stovall's paper \cite{stovall_uniform_2016} proving the optimal uniform result for restriction over real polynomial curves. 

	A key part of these developments has been a geometric lemma by Dendrinos and Wright \cite{dendrinos_fourier_2010} that provides very precise bounds to the torsion in the case of real polynomial curves. At an intuitive level, the geometric lemma states that a polynomial curve can be split into a bounded (depending only on the degree of the polynomial curve) number of open sets so that, in each open set the curve behaves like a \textit{model curve} that is easier to study. This paper extends the result to complex variables using compactness techniques.

	As a consequence this paper extends a theorem by Stovall to the complex case. We prove the restriction theorem  \cite[Theorem 1.1\textbf{}]{stovall_uniform_2016}, therein for complex polynomial curves $\gamma:\mathbb C \to \mathbb C^d$ under the isomorphism $\mathbb C \sim \mathbb R^2$ (Theorem \ref{thm:restriction}). We show, following the proof in \cite{stovall_uniform_2016} that the restriction is uniform over all complex curves, extending results by Bak-Ham in \cite{bak_restriction_2014}, who already provided the uniform restriction result for (complex) dimension $d=3$, and extended it to some particularly \textit{simple} curves in higher dimensions.

	\subsection{The affine measure on a complex curve} 
	\label{sub:the_affine_measure_on_a_complex_curve}

	In this article we will consider the extension/restriction operators with respect to the \emph{complex affine arclength measure}. Before we define the compplex affine arclength measure, let us define its real, more commonly used counterpart. For a real $C^d$ curve $\gamma:[0,1] \mapsto \mathbb R^d$ we define the real affine measure of $\gamma(t)$ as a weighted pushworward of the $dt$ measure:
	
		\begin{equation}\label{eq:def_real_affine_measure}
		d \lambda_\gamma := \frac 1 {d!}  \gamma_* \left(\det [\gamma'(z), \gamma''(z), \dots ,\gamma^{(d)}(z)]^{\frac{2}{d^2+d}} dt\right)
	\end{equation}
	equivalently by duality, for a function $g$ in $C_0(\mathbb R^d;\mathbb R)$,
	
	\begin{equation}
		\int_{\mathbb R^d} g(x) d \lambda_\gamma(x) := \int_0^1 g(\gamma(t)) \left(\det [\gamma'(z), \gamma''(z), \dots ,\gamma^{(d)}(z)]^{\frac{2}{d^2+d}}\right)  dt
	\end{equation}
	
	The potentially suitability of this measure (which vanishes at all the points where the torsion $\det [\gamma'(z), \gamma''(z), \dots ,\gamma^{(d)}(z)]$ vanishes) to control the potential singularities of a restriction estimate starts as early as in \cite{drury_fourier_1985}, and this is the measure used in the main theorem of \cite{stovall_uniform_2016}.

	Inspired by the real affine arclength measure, we define the affine arclength measure associated to a $d-$dimensional complex analytic curve $\gamma$ defined on an open set $D \subseteq \mathbb C$ (i.e $\gamma(z):D \to \mathbb C^d$) as the push-forward of the Lebesgue measure weighed by a power ($\frac 4 {d^2+d}$) of the torsion of the curve:

	\begin{equation}\label{eq:def_affine_measure}
		d \lambda_\gamma = \frac 1 {d!}  \gamma_* \left(\det [\gamma'(z), \gamma''(z), \dots ,\gamma^{(d)}(z)]^{\frac{4}{d^2+d}} |dz|\right)
	\end{equation}
	
	The canonical example is the case when the curve $\gamma(t):= (t,t^2, \dots, t^d)$ is a moment curve. In this case, $d \lambda_\gamma = \gamma_* dz$, because the determinant is equal to 1. In the real smooth setting ($\gamma:\mathbb R \to \mathbb R^d$), this measure has been considered in the literature \cite{drury_fourier_1985,bak_restriction_2014} for two of its properties that transfer readily to the complex case:

	\begin{itemize}
		\item The measure $\lambda$ is then covariant both under local re-parametrization of $z$ (i.e., if $\phi: D'\to D$ is a conformal map $\lambda_{\gamma} = \lambda_{\gamma \circ \phi}$) and affine maps applied on $\mathbb C^d$ (i.e. if $L \in GL(\mathbb C;d)$, then $d\lambda_{L\circ\gamma} = L_* d\lambda_\gamma$). The factor $4$ in the $\frac 4 {d^2+d}$ in the definition \eqref{eq:def_affine_measure} (in comparison with the factor of 2 in \eqref{eq:def_real_affine_measure}) is to preserve this fact.
		
		\item The measure $\lambda$ vanishes at the points where the torsion of $\gamma$ ($\det [\gamma', \gamma^{(2)}, \dots, \gamma^{(d)}]$) of $\gamma$ vanishes. The relevance of this property comes from the fact that the restriction theorem in the full range fails if one uses the the Hausdorff measure at neighbourhood of a point where the torsion of $\gamma$ vanishes.
	\end{itemize}

	Moreover, Bak and Ham \cite{bak_restriction_2014} show that this measure is optimal for the Fourier restriction problem, in the sense that any measure supported on the image of $\gamma$ for which Theorem \ref{thm:restriction} below holds in its full range of exponents must be absolutely continuous with respect to $d\lambda_\gamma$.

	To make things more notationally convenient in the following sections, we will not only consider the  affine measure, but a set of related differential forms on the domain $\mathbb C$ of $\gamma$  , for $0<k\le d$ define:

	\begin{equation}
		\Lambda^{(k)}_\gamma(z) := \gamma'(z)\wedge  \dots \wedge \gamma^{(k+1)}(z)
	\end{equation}
	\begin{equation}
		\Lambda_\gamma(z_1, \dots , z_k) := \gamma'(z_1)\wedge\dots\wedge \gamma'(z_k)
	\end{equation}
	Note that $\Lambda_\gamma$ is a function with variable arity (which will be clear by the context) that has an element of $\mathbb C^k$ as an input and returns a $k-$form as an output. We will denote the Vandermonde determinant by

	\begin{equation}
		v(z_1, \dots z_k) := \prod_{i<j} (z_i-z_j)
	\end{equation}
	and, to be consistent with the previous notation in \cite{stovall_uniform_2016}, we will define 

	\begin{equation}
		L_\gamma(z) := \frac 1 {d!}\Lambda^{(d)}_\gamma(z) = \frac 1 {d!} \det [\gamma'(z), \gamma''(z), \dots ,\gamma^{(d)}(z)]
	\end{equation}.
	
	, which leads to $\lambda_\gamma = L_{\gamma}^{4/{d^2+d}}$. In the case where the arity of $\Lambda$ is equal to $d$, we define $J_\gamma := \Lambda_\gamma$.
	
	\begin{example}[Differential forms associated to curves]
	\label{ex:differential_forms}
	Let $\mu(z) := (z,z^2, z^3)$ be the standard moment curve. Let $e_1, e_2, e_3$ be the canonical co-ordinate basis on the 1-forms $\Lambda^1(\mathbb C^3)$. Then
\begin{align}
        \Lambda^{(1)}_\mu(z) =& e_1 + 2 z e_2 + 3 z^2 e_3    \\
        \Lambda^{(2\textbf{})}_\mu(z) =& \nonumber
        \Lambda^{(1)}_\mu(z)  \wedge  (2 e_2 + 6 z e_3)
        \\=&
        2 e_1\wedge e_2 + 6 z^2 e_2 \wedge e_3 - 6z e_3 \wedge e_1 \\
        \Lambda^{(3)}_\mu(z) =&
        \Lambda^{(2)}_\mu(z)  \wedge  (6d w_3) =
        12 e_1 \wedge e_2 \wedge e_3,
\end{align} and

\begin{align}
    \Lambda_\mu(z) =& e_1 + 2 z e_2 + 3 z^2 e_3    \\
    \Lambda_\mu(z_1, z_2) =& \mu'(z_1)\wedge \mu'(z_2) \nonumber
    \\=&
    (z_2-z_1) (2 e_1 \wedge e_2 + 6  z_1 z_2 e_2 \wedge e_3 -3 (z_1+ z_2) e_3\wedge e_1)\\
    \Lambda_\mu(z_1, z_2, z_3) =& \det (\mu'(z_1), \mu'(z_2), \mu'(z_3)) e_1\wedge e_2 \wedge e_3 =  6 v(z_1, z_2, z_3) e_1\wedge e_2 \wedge e_3
\end{align}

Note that in the preceding example, the form $\Lambda_\gamma(z_1, z_2, z_3)$ (resp. $\Lambda_\gamma(z_1, z_2)$) is divisible by $v(z_1, z_2, z_3)$ (resp. $v(z_1,z_2)$). This is not by chance, if two points $z_i, z_j$ are the same then the vectors associated to them are aligned, and the associated form must vanish.

A more interesting example is the generalized moment curve. Let $\mbf n$ be an increasing multi-index of length $d$. Let $\gamma_{\mbf n}$ be the \emph{generalized moment curve of degree} $\mbf n$,  $\gamma_{\mbf n}:=(z^{\mbf n_1}, \dots ,z^{\mbf n_d})$. Define $\delta(\mbf n) := (\mbf n_1 -1, \mbf n_2 - \mbf n_1 -1, \mbf n_3 - \mbf n_2 - 1, \dots , \mbf n_k - \mbf n_{k-1}-1)$. Define $S_\delta$, the Schur polynomial of index $\delta$, to be the polynomial obeying the identity  $S_\delta(z_1, \dots z_k)v(z_1, \dots z_k) = \det[ (z_i^ {n_j})_{i,j =1, \dots n}]$.
For a given element $\lambda$ of the exterior algebra of order $k$, and a canonical basis element $e_{j_1}\wedge \dots \wedge e_{j_k}$, let denote $\lambda|_e$ the coefficient of $w$ in the canonical basis. Now it follows from the definitions that

\begin{equation}
    \Lambda_{\mu_{\mbf n}}(z_1, \dots, z_k) |_{e_{j_1} \wedge \dots \wedge e_{j_k}} =
    \det[ (z_i^ {n_{j_k}})_{i,k =1, \dots n}]
    =S_{\delta(n_{j_1}, \dots n_{j_k})}(z_1, \dots z_k) v(z_1, \dots z_k)
\end{equation}

As we shall see in section \ref{sub:model_case_generealized_moment_curve} for a general curve $\gamma$ we can recover estimates for $\Lambda^{(k)}(z)$ from estimates for $\Lambda(z_1, \dots z_k)$, using the equality $\Lambda^{(k)}(z) = c_{k} \left[\frac{\Lambda(z_1,\dots z_k)}{v(z_1, \dots z_k)}\right](z,\dots_k, z)$ with $c_{k}=\prod_{i=1}^{k-1} k!$. The fact that the form $\left[\frac{\Lambda(z_1,\dots z_k)}{v(z_1, \dots z_k)}\right]$ extends continuously to the zero set of $v(z_1, \dots , z_k)$ for general curves other than the generalized moment curve (\ref{lem:det_is_limit}) will become a core part of the proofs in section \ref{sub:fixed_polynomial_case}.

	\end{example}


	\subsection{Restriction Problem} 
	\label{sub:restriction_problems}
	
	The main Fourier restriction result of this article is inspired the following uniform restriction theorem:
	
	\begin{thm}	[{\cite[Theorem 1.1]{stovall_uniform_2016}}]	
		\label{thm:real_restriction}
		For each $N,d$ and $(p,q)$ satisfying:

		\begin{equation}
			p' = \frac{d(d+1)}{2} q, \;\; q> \frac{d^2+d+2}{d^2+d}
		\end{equation}
		there is a constant $C_{N,d,p}$ such that for all polynomials $\gamma: \mathbb R \to \mathbb R^d$ of degree up to $N$ we have:

		\begin{equation}
			\|\hat f\|_{L^q(d\lambda_\gamma)} \le C_{N,d,p} \|f\|_{L^p(dx)}
		\end{equation}
		for all Schwartz functions $f:\mathbb R^d \to \mathbb R$, where $d\lambda_\gamma$ is the real affine arclength measure defined in \eqref{eq:def_affine_measure}.
	\end{thm}
	We will prove the complex analogue of the theorem above, using the complex arclength measure, and with the same $p,q$ exponent range, that is:
    
	\begin{thm}
	\label{thm:restriction}
		For each $N,d$ and $(p,q)$ satisfying:

		\begin{equation}
			p' = \frac{d(d+1)}{2} q, \;\; q> \frac{d^2+d+2}{d^2+d}
		\end{equation}
		there is a constant $C_{N,d,p}$ such that for all polynomials $\gamma: \mathbb C \to \mathbb C^d$ of degree up to $N$ we have:

		\begin{equation}
			\|\hat f\|_{L^q(d\lambda_\gamma)} \le C_{N,d,p} \|f\|_{L^p(dx)}
		\end{equation}
		for all Schwartz functions $f:\mathbb C^d \to \mathbb C$, where the Fourier transform is the $\mathbb R^{2n}$-dimensional Fourier transform under the $\mathbb C^n \sim \mathbb R^{2n}$ isomorphism, and $d\lambda_\gamma$ is the complex affine arclength measure defined in \eqref{eq:def_affine_measure}.
	\end{thm}

	Bak-Ham provide  a partial answer to the complex restriction theorem  for particular complex curves in \cite{bak_restriction_2014}, and show the optimality of the measure $\lambda$, that is, that any other measure supported on $\gamma$ for which Theorem \ref{thm:restriction} holds in the given range must be absolutely continuous with respect to $\gamma$ \cite[Proposition 2.1]{bak_restriction_2014}, and indicate that the result could hold over all polynomial curves \cite[Footnote 2]{bak_restriction_2014}. 

	Optimality in theorem \ref{thm:real_restriction} was proven in \cite{arkhipov_trigonometric_2008}. We do not have such a result in the complex case for $d>2$. As noted in \cite{bak_restriction_2014}, the optimal exponent in the complex case is unknown, and related to the multi-dimensional Tarry problem.

	This paper follows the proof in \cite{stovall_uniform_2016}. The main challenge is finding a complex substitute for Lemma 3.1 from that paper (which we do in section \ref{sec:the_geometry}). Once this matter is resolved, section \ref{sec:the_analysis} follows the argument in \cite{stovall_uniform_2016} closely, elaborating on the small modifications whenever necessary.

	A more amenable version of the problem above is the extension problem, instead of bounding the restriction operator defined above, we bound the dual extension operator $\mathcal E_\gamma: L^p(d\lambda_\gamma) \to L^q(\mathbb R^n)$ defined as

	\begin{equation}
		\mathcal E_\gamma(f) = \mathcal F^{-1}(f d\lambda_\gamma).
	\end{equation}
	Theorem \ref{thm:restriction} now becomes:
	\addtocounter{thm}{-1}
	\begin{thm}[Dual version]
		For each $N,d$ and $(p,q)$ satisfying:

		\begin{equation}
			p' = \frac{d(d+1)}{2} q, \;\; q> \frac{d^2+d+2}{d^2+d}
		\end{equation}
		there is a constant $C_{N,d,p}$ such that for all polynomials $\gamma: \mathbb C \to \mathbb C^d$ of degree up to $N$ we have:

		\begin{equation}
			\|\mathbb E f\|_{L^{p'}(dx)} \le C_{N,d,p} \| f\|_{L^{q'}(d\lambda_\gamma)} 
		\end{equation}
		for all Schwartz functions $f$, where the Fourier transform is the $\mathbb R^{2n}$-dimensional Fourier transform.
	\end{thm}


\subsection{Convolution estimates}

A problem that bears some similarities with the restriction problem that of finding convolution estimates for some complex curves. The problem asks, given a measure $d\lambda_\gamma$, for which $p$ and $q$ is the map $f \mapsto f\ast d\lambda_\gamma$ continuous from $L^p$ to $L^q$.

Again, one may ask whether the norm of the operator depends on the specific polynomial, or only on the degree and dimension. This questions have satisfactory answers in the real polynomial setting, which use the geometric lemma of Dendrinos and Wright \cite{stovall_endpoint_2010}, and partial answers for simple polynomial curves in the complex case \cite{chung_convolution_2019}. The author thinks that the methods developed in this paper can therefore be applicable to an extension of those results to the complex case.

	\subsection{Outline of the work} 
	\label{sub:outline_of_the_work}

	The strategy to prove Theorem \ref{thm:restriction} is as follows:

	\begin{enumerate}
		\item The complex plane $\mathbb C$ is partitioned into a bounded number of sets $\mathbb C = \bigsqcup_{i=0}^{N(n,d)} U_i$, so that for each $U_i$ there is a generalized moment curve $\gamma_i$ so that $\gamma_i \sim_{n,d} \gamma$ in $U_i$. The meaning of $\sim$ will become clear in the following sections). This is the main new technical contribution of this paper, and the goal of Section \ref{sec:the_geometry}. 

		\item The rest of the proof follows (with small, suitable variations) an argument by Stovall:

		\begin{enumerate}
		 	\item Theorem \ref{thm:restriction} is proven for the particular case of (perturbations of) the standard moment curve (Theorem \ref{thm:uniform_local_restriction}).

			\item Then it suffices to prove the respective theorem for perturbations of generalized moment curve. The sense in which we will take perturbations is made precise in section \ref{sec:the_geometry}. The proof of the restriction conjecture for perturbations of the moment curve is done in two steps:
            
            \begin{enumerate}
			\item The domain $\mathbb C$ of the generalized moment curve is split into dyadic scales. The non-degenerate, standard moment curve implies that the restriction theorem holds at each scale.
			
			\item Then an almost-orthogonality result between the scales is proven, that allows to sum back. The argument is the same as in Stovall's \cite{stovall_uniform_2016}, except for modifications outlined in Section \ref{sub:almost_orthogonality}.
			\end{enumerate}
		 \end{enumerate} 
	\end{enumerate}

	
\subsection{Notation}

For arbitrary objects (polynomials, natural numbers..) such as $x,y,z$ we will write $A\lesssim_{x,y,z} B$ to denote that there exists a constant $K$ depending only on $x,y,z$, the ambient dimension $d$ and the degree of the polynomial curves involved $n$ so that $A \le K B$. We will use $A\sim_{x,y,z} B$ to denote $A\lesssim_{x,y,z} B$ and $B\lesssim_{x,y,z} A$. We will also use the notation $A=O_{x,y,z}(B)$ to denote $A\lesssim_{x,y,z} B$.

For a differential form $w \in \Lambda^k(\mathbb C^ n)$ we will write $\|w\|$ to denote the euclidean norm of $w$ in the canonical basis $e_1,\dots, e_n$. Given an element $e := e_{j_1}\wedge \dots \wedge e_{j_k}$ we will denote by $w|_e$ as the $e^{\operatorname{th}}$ co-ordinate of $w$ in the canonical basis. 

\subsection*{Acknowledgements}

The author thanks Terence Tao for his invaluable guidance and support, and Sang Truong for many insightful discussions. This work was supported by ``La Caixa'' Fellowship LCF/BQ/AA17/11610013.

	\newpage
	\section{The partitioning lemma} 
	\label{sec:the_geometry}

	The main goal of this section is to prove a complex analogous to the following theorem (Lemma 3.1 in \cite{stovall_uniform_2016}, originally from \cite{dendrinos_uniform_2013}):
	\begin{lemma} \label{L:polynom decomp}
    Let $\gamma:\mathbb R \to \mathbb R^d$ be a polynomial of degree $N$, and assume that $L_\gamma \not\equiv 0$.  We may decompose $\mathbb R$ as a disjoint union of intervals,
    $$
    \mathbb R = \bigcup_{j=1}^{M_{N,d}} I_j,
    $$
    so that for $t \in I_j$,
    \begin{equation} \label{E:L sim}
    |L_\gamma(t)| \sim A_j|t-b_j|^{k_j}, \qquad |\gamma_1'(t)| \sim B_j |t-b_j|^{\ell_j},
    \end{equation}
    and for all $t = (t_1,\ldots,t_d) \in I_j^d$, 
    \begin{equation} \label{E:dw gi}
    |\frac{J_\gamma(t)}{v(t)}| \gtrsim \prod_{j=1}^d |L_\gamma(t_j)|^{\frac 1d}
    \end{equation}
    where for each $j$, $k_j$ and $\ell_j$ are integers satisfying $0 \leq k_j \leq dN$ and $0 \leq \ell_j \leq N$, and the centers $b_j$ are real numbers not contained in the interior of $I_j$.  Furthermore, the map $(t_1,\ldots,t_d) \mapsto \sum_{j=1}^d \gamma(t_j)$ is one-to-one on $\{t \in I_j^d : t_1 < \cdots < t_d\}$.    The implicit constants and $M_{N,d}$ depend on $N$ and $d$ only.  
    \end{lemma}

    The main difference when in the statement for the complex case comes from defining the concept of interval. Our alternative concept is that of generalized triangle. Generalized triangles are the (possibly unbounded) intersection of three semiplanes. The second main difference is the fact that we will not be able to prove that the sum map is one-to-one. We will instead prove it has bounded generic multiplicity (see \ref{lem:geometric_lemma} for the precise statement). This fact, as already observed in \cite{bak_restriction_2014}. The torsion $L^\gamma:=\Lambda^{(d)}$ is a particular case of the $\Lambda_\gamma^{(k)}$ that we will need to resort to in multiple occasions in the sequel. Therefore we present the complexified version with that notation instead.

	\begin{thm}[Lemma \ref{L:polynom decomp} complexified]
	\label{lem:geometric_lemma}

		Let $\gamma:\mathbb C \to \mathbb C^d$ be a polynomial curve of degree $N$, and assume $\Lambda^{(d)}_\gamma$ is not identically zero. Then we can split $\mathbb C \cup \{\infty\}$ into $M = O_{N}(1)$ generalized triangles\footnote{The generalized triangles in $\mathbb C \sim \mathbb R^2$ we will consider are open sets of the form $\{ x \in  \mathbb R^2| v_i\cdot x>0, i=1,2,3\}  $ for three given vectors $v_i\in\mathbb R^2$. This includes the (interior) convex hull of three points, but also triangles with a vertex at infinity, or half planes.}  $$C = \bigcup^{M_{N,d}}T_j$$ so that on each generalized triangle  $T_j$:
		\begin{equation}
		\label{eq:power_geometric_estimates}
			|\Lambda_{\gamma}^{(d)}(z)| \sim_N A_j|t-b_j|^{k_j}, \qquad |\gamma_1'(t)| \sim_N B_j |t-b_j|^{l_j},
		\end{equation}
		and, for all $z := (z_1, \dots, z_d) \in T_j^d$:
		\begin{equation}
			\tag{DW}
			\label{eq:hard_jacobian_estimate}
			\left|\frac{J_{\gamma}(z)}{v(z)}\right|\gtrsim_N \prod_{i=1}^d \Lambda^{(d)}_{\gamma}(z_i)^{1/d}
		\end{equation}
		where for each $j$ $k_j$, $l_j$  are integers satisfying $0 \leq k_j \leq dN$, and the centers $b_j \in \mathbb C$ are not in $T_j$.
		Furthermore, for each triangle $T_j$ there is a closed, zero-measure set $R_j \subseteq T_j^d$ so that the sum map $\Sigma(z):=\sum_{i=1}^d \gamma(z_i)$ is $O_N(1)$-to-one in $T_j^d\setminus R_j$. The implicit constants and $M_{N,d}$ depend on $N$ and $d$ only.  
	\end{thm}

    \begin{example}
    To understand the need to split $\mathbb C$ into multiple sets, consider the curve $\gamma(z):= (z, 2 z^3 - 3 z^2$. We have that 
    $$\Lambda_\gamma(0,1) = \det \begin{pmatrix}1 &0 \\ 1& 0 \end{pmatrix} = 0$$
    which shows that the theorem does not hold unless we partition $\mathbb C$ in such a way that $0$ and $1$ do not belong to the same subset. 
    
    As we shall see, for any point  $z\neq \frac 1 2$ (the only point where $\Lambda^{(2)}$ vanishes) there is a neighbourhood $B_z$ where $\gamma$ behaves similarly the standard moment curve. Lemma \ref{lem:geometric_lemma} is trivially true for the moment curve (for the moment curve case, estimate \eqref{eq:hard_jacobian_estimate}, the hardest estimate in \ref{lem:geometric_lemma} reads $1\gtrsim 1$), and proposition \ref{prop:local_nondegenerate_jacobian} shows how to use this fact to transfer it to curves that behave similarly.
    
    Therefore to prove estimate \eqref{eq:hard_jacobian_estimate} of Theorem \ref{lem:geometric_lemma} it suffices to show that it holds near $z=\frac 1 2$, and at infinity.
    
    By saying that the estimate holds at infinity we mean that for a big enough ball $B_R$ we can split $\mathbb C\setminus B_R$ into $O(1)$ generalized triangles where \eqref{eq:hard_jacobian_estimate} holds. The idea is similar. For $|z|>R$ $\gamma(z) = (z, z^3(2+3 z^{-1}) )$ behaves like a moment curve $\mu_{(1,3)} = (z, z^3)$ (by the affine invariance the factor 2 does not change the statement). Section\ref{sub:model_case_generealized_moment_curve} shows that estimate \eqref{eq:hard_jacobian_estimate} holds for generalized moment curves, and section \ref{sub:uniformity_for_polynomials} explains how to transfer that fact to perturbations of moment curves at infinity.
    
    Section \ref{sub:fixed_polynomial_case} deals with zeroes of $\Lambda^{(2)}_\gamma$ (the zeroes of $\Lambda^{(d)}_\gamma$ in the general case) at finite points. The approach is the same. The curve $\tilde \gamma(z) := \gamma(z+\frac 1 2) -\gamma(\frac 1 2)= (z, 2z^3 +6z) $ behaves like $\mu_{1,3}$ near zero. This is because an affine transformation lets us cancel the term $6z$.
    \end{example}

	This theorem is the main geometric ingredient that is used in the real version of Theorem \ref{thm:restriction}. Informally, the theorem states that locally a polynomial is similar to a moment curve. Estimate \eqref{eq:power_geometric_estimates} was proven in \cite{bak_restriction_2014}, and can be proven using suitable modifications of the argument described in Lemma 3.1 of \cite{stovall_uniform_2016} as well. To simplify notation, and since the inequality we want to prove above concerns $\gamma'$ only and not $\gamma$, we will prove inequality \eqref{eq:hard_jacobian_estimate} for a generic curve $\gamma$ that will end up being the $\gamma'$ above.

	The strategy to prove the theorem will be the following: First inequality \eqref{eq:hard_jacobian_estimate} will be shown for the moment curve, or for generalized moment curves (that is, curves that are affine equivalent to a curve of the form $(z_1^{\delta_1}, \dots, z_1{^{\delta_d}})$, see example \ref{ex:differential_forms} above). Then, we will show  that the result is in fact stable to suitable small perturbations of the polynomial. This, together with a compactness argument on $\mathbb C \cup \{\infty\}$, will give the non-uniform estimate (where the constant could depend on the polynomial). The only potential source of non-uniformity at this point will be the number of open sets, and finally we will show a stability on the number of open sets to use, and a compactness argument on the set of polynomials with coefficients $\lesssim 1$ will finish the proof.

	\subsection{Preliminaries} 
	\label{sub:prelim}

	In this section we will define a systematic way of changing co-ordinates to polynomial curves to understand the behavior near a point, which we refer to as \textit{the zoom-in method} from now on. 

	\begin{defi}
		[Canonical form for a curve]
		Let $\gamma = (\gamma_1, \dots, \gamma_d)$ be a polynomial curve of dimension $d$. Let $\delta_i$ be the lowest degree of a non-zero monomial in $\gamma_i$. Then $\gamma$ is in \emph{canonical form at zero} if:

		\begin{itemize}
		    \item $\gamma(0)=0$
			\item $\delta_1< \dots < \delta_d$
			\item The coefficient of degree $\delta_i$ in $\gamma_i$ is $1$.
		\end{itemize}
		similarly, that $\gamma$ is in canonical form at $c\in\mathbb C$ if the translated curve $z\mapsto \gamma(z-c)$ is in canonical form at zero. The curve $\gamma$ is in canonical form at infinity if $z^D \gamma(z^{-1})$ \textbf{is} in canonical form at zero, where $D$ is the maximum of the degrees of $\gamma_i$.
	\end{defi}
	
	\begin{example}
	Let $\mu(z)=(z, z^2, \dots, z^d)$ be the moment curve. Then $\mu$ is in canonical form at zero (and the same is true for any generalized moment curve $\mu_{\mbf n}= (z^{\mbf n_1}, \dots , z^{\mbf n_d})$). In general, curves of canonical forms are higher order perturbations of generalized moment curves, and that is what makes them relevant.
	\end{example}

	By performing a Taylor expansion it is not difficult to see that polynomial admits a canonical form (after an affine reparametrization) for every point if and only if the Jacobian for the curve is not the zero polynomial (that is, as long as the curves are linearly independent as polynomials). Given a polynomial $\gamma$ in canonical form we define a \textbf{zoom in at zero at scale $\lambda$} as the (normalized) zoom-in:

	\begin{defi}[Zoom in on a curve]
	 Given a curve $\gamma$ in canonical form, the \emph{zoom-in} of $\gamma$ \emph{at scale $\lambda$} is the polynomial curve $\mathcal B_\lambda [\gamma](z) := \text{diag}(\lambda^{-\delta_1}, \dots, \lambda^{-\delta_d}) \gamma(\lambda z)$. 
	\end{defi}
	
	Note that the coefficients of $\mathcal B_\lambda [\gamma](z)$ converge to the polynomial $(z^{\delta_1}, \dots, z^{\delta_d})$ as $\lambda$ goes to zero. 

	
	\subsection{Model Case: Generalized moment curve} 
	\label{sub:model_case_generealized_moment_curve}

	For a generalized moment curve $\mu_{\mbf n}$ with exponents $\mbf n:= (\mbf n_1< \dots <\mbf n_d)$ (that is $\mu_{\mbf n}(z):= (z^{n_1}, \dots, z^{n_d})$ ), and for ${\bf z} := (z_1, \dots, z_d)$ the following holds:

	\begin{equation}
		\frac{J(\mbf z)}{v(\mbf z)} = \pi_{\mbf n}S_{\delta(\mbf n)} (\bf z)
	\end{equation}
	where $\pi_{\mbf n}$ is the product of all the $n_i$, $\delta(\mbf n) = (\mbf n_1-1, \mbf n_2-2, \dots, \mbf n_d-d)$ is the excess degree, and $S_{\mbf k}$, for a general non-decreasing multi-index $k$ is the Schur polynomial associated to $\mbf k$. A classical result in algebraic combinatorics (see, for example, \cite{fulton_1996}) states that:
	
	\begin{equation}
	S_{\bf k}(z_1, \dots, z_d) = \sum_{(t_i) \in T_{\bf k}} z_1^{t_1},\dots, z_d^{t_d}	
	\end{equation}	
	where $T_{\bf k}$ is the set of semistandard Young tableaux of shape $\bf n$, and $(t_i)$ are the weights associated to each tableaux $t$.
	
    \begin{example}
        Let $\mbf n= (1,3,5)$. Then the excess degree is $(0,1,2)$. The Young diagram associated to this partition is of the form
        
$$\begin{tabular}{|c|c|}
\hline
 *&  *\\
\cline{1-2}
 * \\
\cline{1-1}
\end{tabular}$$ where the row of length 0 is not drawn. Thus, counting over all possible semi-standard Young Tableaux (ways of filling the Young diagram with the indices $1,2,3$ ($1,\dots ,d$ in general) that are strictly increasing vertically and weakly increasing horizontally) gives the polynomial:
   
\[
\begin{tabular}{ccccccccc}
&
\begin{tabular}{|c|c|}
\hline
$1$ & $1$ \\
\cline{1-2}
$2$ \\
\cline{1-1}
\end{tabular}\ &
\begin{tabular}{|c|c|}
\hline
$1$ & $1$ \\
\cline{1-2}
$3$ \\
\cline{1-1}
\end{tabular}\ &
\begin{tabular}{|c|c|}
\hline
$1$ & $2$ \\
\cline{1-2}
$2$ \\
\cline{1-1}
\end{tabular}\ &
\begin{tabular}{|c|c|}
\hline
$1$ & $2$ \\
\cline{1-2}
$3$ \\
\cline{1-1}
\end{tabular}\ &
\begin{tabular}{|c|c|}
\hline
$1$ & $3$ \\
\cline{1-2}
$2$ \\
\cline{1-1}
\end{tabular}\ &
\begin{tabular}{|c|c|}
\hline
$1$ & $3$ \\
\cline{1-2}
$3$ \\
\cline{1-1}
\end{tabular}\ &
\begin{tabular}{|c|c|}
\hline
$2$ & $2$ \\
\cline{1-2}
$3$ \\
\cline{1-1}
\end{tabular}\ &
\begin{tabular}{|c|c|}
\hline
$2$ & $3$ \\
\cline{1-2}
$3$ \\
\cline{1-1}
\end{tabular}
\\
$S_{(0,1,2)}$
& $= \quad z_1^2z_2$ 
& $+ \quad z_1^2z_3$
& $+ \quad z_1z_2^2$
& $+ \quad z_1z_2z_3$
& $+ \quad z_1z_2z_3$
& $+ \quad z_1z_3^2$
& $+ \quad z_2^2z_3$
& $+ \quad z_2z_3^2$
\end{tabular}
\]

\end{example}
	
The statement above gives us a lot of control about the $\Lambda_{\mu_(\mbf n)}$ for general moment curves. The most important fact will be that they are symmetric polynomials with positive coefficients times the Vandermonde term. Now, to compare the  forms $\Lambda^{(d)}_{\mu_(\mbf n)}(z)$ and $\Lambda(z_1, \dots, z_d)$, the following fact is useful:

	\begin{lemma}
	\label{lem:det_is_limit}
		Let ${\bf z}\in \mathbb C^d$, with $z_i\neq z_j$ for $i\neq z$, let $s\in \mathbb C$, and $\gamma:\mathbb C\to \mathbb C^d$ a polynomial curve, then:
		\begin{equation}
		J_\gamma(s) = C_d \lim_{\lambda\to 0} \frac{\Lambda_\gamma(\lambda {\bf z}+ s)}{v(\lambda {\bf z})} 
		\end{equation}  
		and, in particular, in the case when $\gamma$ is a moment curve of exponent $\bf n$,
		\begin{equation}
		\label{eq:det_is_schur}
			\Lambda^{(d)}_\gamma(s) = C_d S_{\bf n}(s, \dots, s)
		\end{equation}
	\end{lemma}

	\begin{proof}
		By Taylor expansion we have:

		\begin{equation}
			\gamma'_i(s+ \lambda \mbf z_j) =  \sum_{k=1}^{d} \frac 1 {(k-1)!}\gamma_i^{(k)}(s) \lambda^{k-1} \mbf z_j ^{k-1} + O(\lambda^d).
		\end{equation}
		Now, defining the matrices $\Gamma_{ij}' = \gamma_i(s+ \lambda \mbf z_j)$, $ { Z}_{kj} =  (\lambda \mbf z_j) ^{k-1}$, and $(T_\gamma)_{ik} = \frac 1 {(k-1)!}\gamma_i^{(k)}(s) $ the equation above can be rewritten as:

		\begin{equation}
			\Gamma = T_\gamma Z + O(\lambda^d).
		\end{equation}
		Since the determinant of $Z$ is $v({\lambda \mbf z})$, the lemma follows from the multiplicative property of the determinant:

		\begin{equation}
			\frac{\Lambda_\gamma(\lambda {\bf z}+ s)}{v(\lambda {\bf z})} = 
			\frac{\det \Gamma}{\det Z} =  \det [T_\gamma +  Z^{-1}O(\lambda^d)] \stackrel{\lambda \to 0}{\to} \det T_\gamma =  C_d \Lambda^{(d)}_\gamma(s).
		\end{equation}
		The fact that $Z^{-1}= o(\lambda^{-d})$ (and thus we can eliminate the term as $\lambda\to 0$) is a quick computation from the adjoint formula for the inverse. Note that the value of $C_d = \prod_{i=1}^d \frac 1 {(k-1)!}$ is in fact explicit.
	\end{proof}

	\begin{remark}
		The same argument works as well for $\Lambda^{(k)}_\gamma$, $1\le k<d$, since each component of $\Lambda^{(k)}_\gamma$ is a determinant of components of the polynomial, giving to the more general equality:
		
		\begin{equation}
		\Lambda^{(k)}(s) = C_k \lim_{\lambda\to 0} \frac{\Lambda_\gamma(\lambda {\bf z}+ s)}{v(\lambda {\bf z})} 
		\end{equation}  
		
		for $\bf z$ in $\mathbb C^k$ without repeated components.
	\end{remark}
	
	The most relevant example of the theorems above is the case where $\mu_{\mbf n}$ is a moment curve. In this case (see Example \ref{ex:differential_forms}) Lemma \ref{lem:det_is_limit} implies that the torsion $\Lambda^{(d)}$ is (up to a constant depending on $\mbf n$ and $d$) equal to the Schur polynomial $S_{\delta(\mbf n)}$ (where, again $\delta(\mbf n) = (\mbf n_1 - 1, \mbf n_2 -2, \dots , \mbf n_d - d)$. A fact that will be extremely relevant is that all the monomials of Schur polynomials are positive. Our application of this fact is the following lemma:
    
    \begin{lemma}
    \label{lem:schur_positive_complex}
        Let $S_{\mbf n}$ be a Schur polynomial in $d$ variables, then there exists an $\epsilon>0$ depending on $\mbf n,d$ so that for any angular sector $W_i = \{z \in \mathbb C \setminus\{0\}: |\arg z - \theta|<\epsilon \}$ and complex numbers $z_1, \dots, z_d \in W$ 
        
        \begin{equation}
        \label{eq:schur_positive_complex}
            \left|\frac{e^{-i |\mbf n| \theta} S_{\mbf n}(z_1, \dots, z_d)}{ S_{\mbf n}(|z_1|, \dots, |z_d|)} -1 \right |< \frac 1 {100}
        \end{equation}
        
        in particular, we obtain that for $z =(z_1, \dots z_k)$ in $W_i^k$ we have a \emph{reverse triangle inequality} for Schur polynomials:
        
        \begin{equation}
           \left |  \sum_{t\in T_n} z^t \right| = |S_n(z)|\gtrsim_{n,d} S_n(|z|) = \sum_{t\in T_n} |z^t|
        \end{equation}
    \end{lemma}
    
    \begin{proof}
        First note that $S_n(e^{i\theta}z_1, \dots ,e^{i\theta}z_d) = e^{i|\mbf n|\theta} S_n(z_1, \dots ,z_d)$. Therefore estimate \eqref{eq:schur_positive_complex} is rotationally invariant and we can assume $\theta=0$. Then note that $S_{\mbf n}$ is a positive combination of monomials, and therefore it suffices to show that for any monomial $m$ appearing in $S_{\mbf n}$ we have that $$\left|\frac{m (z_1, \dots, z_d)}{m(|z_1|, \dots, |z_d|)} -1 \right |< \frac 1 {100K}$$ for $K$ equal to the sum of all the coefficietns of the monomials of $S_{\mbf n}$. Since the monomial $m$ has degree $|n|$, for some $\tau \in [-\epsilon |n|, \epsilon |n|]$ we have $\frac{m (z_1, \dots, z_d)}{m(|z_1|, \dots, |z_d|)} = e^{i\tau}$. Making $\epsilon>0$ small enough finishes the proof. 
    \end{proof}

	The lemma above will be used in multiple times in the sequel, and so will the intuition behind the lemma: an important tool to control a quotient of two polynomials is to show that the terms on the denominator don't have much cancellation. Lemma \ref{lem:det_is_limit} and the fact that Schur polynomials have positive coefficients (in the form of Lemma \ref{lem:schur_positive_complex}) are all we need to show the theorem for the generalized moment curve: 

	\begin{proof}
		[Proof (of \eqref{eq:hard_jacobian_estimate}, generalized moment curve case).]

		Let $\mu$ be a generalized moment curve of exponents $\mbf n$. Decompose $\mathbb C = \bigcup W_i$ into finitely many sectors $W_i = \{z: |\arg z - \theta_i|<\epsilon \}$ of aperture $\epsilon$ small enough (depending on the exponents). Now, for $\mbf z = (\mbf z_1, \dots, \mbf z_d) \in W_i^d$

		\begin{equation}
			 |S_{\bf n}(\mbf z)| 
			 \gtrsim_{\mbf n}
			 |\mbf z_1\cdot \mbf z_2 \cdot \dots \cdot \mbf z_d|^{\frac{\deg S_{\bf n}}{d}}
			 = K_{\bf n} \left|\prod_{i=1}^d S_{\bf n}(\mbf z_i, \dots , \mbf z_i) \right|^{1/d}
		\end{equation}
		
		for some $K_{\mbf n}>0$. The first inequality is AM-GM inequality for all the monomials of $S_{\bf n}(\mbf z)$. The second equality follows from the fact that $ S_{\mbf n}(\mbf z_i, \dots , \mbf z_i) = C_{\mbf n} \mbf z_i^{{\deg S_{\mbf n}}}$ for some positive integer $C_{\mbf n}$. Now the result follows from equation \eqref{eq:det_is_schur} from Lemma \ref{lem:det_is_limit}.
	\end{proof}

	Lemma \ref{lem:det_is_limit} leads to the definition of a new differential form that corrects for the Vandermonde factor:

	\begin{defi}
		[Corrected multilinear form]
		For $\gamma:\mathbb C \to \mathbb C^n$ and $\mbf z \in \mathbb C^n$ we define:
		\begin{equation}
			\tilde \Lambda_\gamma(\mbf z) = 
			\frac
			{\Lambda_\gamma(\mbf z) }
			{v(\mbf z) }
		\end{equation}
		moreover, (as we shall see in the following section) the map ${\tilde\Lambda}_{\cdot}(\cdot)$ is continuous in its domain $ \mathbb C^d \times P_N(\mathbb C)^d$.
	\end{defi}

    \begin{example}
    Going back to the moment curve in dimension 3 (see Example \ref{ex:differential_forms}), we see that $$\tilde \Lambda_\mu(z_1,z_2) = 2d e_1 \wedge d e_2 + 6  z_3 z_2 d e_2 \wedge d e_3 -3 (z_1+ z_2) d e_3\wedge d e_1$$ (which, for $z_1=z_2$ happens to be equal to $\Lambda^{(2)}_\mu(z)$, as expected from Lemma \ref{lem:det_is_limit} with the multiplicative $C_2 = \frac 1 {0!1!} = 1$).
    
    The next relevant example (see Example \ref{ex:differential_forms} again) is the case of generalized moment curves $\mu_{\mbf n}$. In that case the coefficients of $\tilde \Lambda_{\mu_{\mbf n}}(z_1, \dots, z_k)$ are Schur polynomials. In particular, the coefficient associated to the basis element $e_{j_1}\wedge\dots \wedge e_{j_k}$ is (up to a constant $\prod_{i=1}^ k \mbf n_{j_i}$ arising from the differentiating the monomials $z_i^{\mbf n_i}$) the Schur polynomial $S_{\delta(\mbf n_{j_1},\dots , \mbf n_{j_k})}$.
    \end{example}

	The first application of this fact is an extra property of the generalized moment curve that will be used later in this section, that we will prove now:

	\begin{lemma}
		[Transversality of the corrected multilinear form for moment curves]
		\label{lem:transversality_corrected} Let $\mu$ be a moment curve, and let $W$ be a sector of $\mathbb C$ of aperture $\epsilon>0$ (depending on $\mu$) small enough. Let $\{\mbf w^{(j)}\}_{j=1}^{\infty} $ be a sequence of elements ${\bf w}^{(k)} = ({\bf w}_1^{(j)}, \dots, {\bf w}_s^{(j)})$ in $ W^s$, let $ \{\mbf z^{(j)} \}_{j=1}^{\infty}$ a sequence in $ W^t$, with $k:=s+t\le d$, assume $|\mbf z^{(j)}_i| = O(1)$, and $\mbf w^{(j)} \to 0$.

		\begin{equation}
		\label{eq:transversality}
			\|\tilde \Lambda_\mu (\mbf z^{(j)}_1, \dots, \mbf z^{(j)}_t, \mbf w^{(j)}_1 ,\dots, \mbf w^{(j)}_s) \| \approx_\mu 
			\|\tilde \Lambda_\mu (\mbf z^{(j)}_1, \dots, \mbf z^{(j)}_t)\|\|\tilde \Lambda_\mu( \mbf w^{(j)}_1, \dots, \mbf w^{(j)}_s) \|
		\end{equation}
		as $j \to \infty$.
	\end{lemma}

	\begin{proof}
		The $\lesssim_\mu$ direction in \eqref{eq:transversality} follows from the fact that, for forms $\|a\wedge b\| \le \|a\|\|b\|$. 

		For the converse, let $e_{n_1} \wedge e_{n_2} \wedge \dots \wedge e_{n_k}$ be one of the co-ordinates on the LHS of \eqref{eq:transversality}. By restricting the curve to those co-ordinates, it can be assumed that $k=d$ (and the result will follow by summing for each co-ordinate). Note also that the term $\|\tilde \Lambda_\mu (z^{(k)}_1, \dots, z^{(k)}_t)\|$ is a the absolute value of a Schur polynomial with all coefficents on a sector of small aperture (and therefore is $O(1)$ by \ref{lem:schur_positive_complex}), so we can omit it in the estimates. 

		By using the Young tableau decomposition of the Schur polynomials again, it suffices to show that each monomial in any of the coordinates $\| \tilde \Lambda_\mu( w^{(k)}_1, \dots, w^{(k)}_s) \|$ is dominated by a monomial  of $\tilde \Lambda_\mu (z^{(k)}_1, \dots, z^{(k)}_t, w^{(k)}_1, \dots, w^{(k)}_s)$. By positivity (using the same arguments as in Lemma \ref{lem:schur_positive_complex}) we can assume without loss of generality that all $z_i, w_i$ are positive real numbers using the \emph{reverse triangle inequality} of \ref{lem:schur_positive_complex}.
		
		We will motivate the proof of domination at the monomial level by an example.
		\begin{example}
		Assume $t=2, s=3$, and $\gamma=(t,t^2, t^4, t^6,t^7)$. That means the associated tableaux will have shape $(0,0,1,2,2)$.
		
  		Let's look at the element $e_2\wedge e_3\wedge e_4$. The curve $\gamma$ restricted to the co-ordinates $(2,3,4)$ is $(t^2,t^4,t^6)$ that has associated tableaux $(1,2,3)$. The important fact here is that elementwise $(1,2,3)$ is bigger than the smallest three rows of the diagram for the full dimension $(0,0,1)$. Therefore, given a tableau (for example):

        $$\begin{tabular}{|c|c|c|}
        \hline
         $w_3$ & \cellcolor{Gray}  $w_2$ & \cellcolor{Gray} $w_1$\\
         \cline{1-3}
         $w_2$ \cellcolor{Gray} & \cellcolor{Gray}$w_1$ \\
         \cline{1-2}
         $w_1$ \cellcolor{Gray} \\
        \cline{1-1}
        \end{tabular}$$ 
        we can now remove the shaded squares to turn it into a $(0,0,1)$ tableau. We can extend this tableau
        to a $(0,0,1,2,2)$ tableau with an associated monomial that dominates the monomial $w_1^3 w_2^2 w_3$:
        
        $$\begin{tabular}{|c|c|}
        \hline
         $z_2$ & $z_2$\\
         \cline{1-2}
         $z_1$  & $z_1$ \\
         \cline{1-2}
         $w_3$ \\
        \cline{1-1}
        \end{tabular}$$ 
        
        \end{example}

        The general proof is analogous to the example: given two Young diagrams $T,T'$ of the same height (counting the rows of length zero) we say $T\le T'$ if each row of $T$ has at most as many elements as the row in $T'$. Let $T$ be the diagram coming from a component $e_{i_1} \wedge \dots \wedge e_{i_t}$ (with row lengths $\mbf n_{i_1}- 1, \dots \mbf n_{i_t}-t$) and let $T'$ be the bottom $t$ rows of the diagram associated to the full determinant $e_{1} \wedge \dots \wedge e_{k}$ (which has lengths $\mbf n_{1}- 1, \dots \mbf n_{t}-t$). The strict monotonicity of $n_i$ implies that $T\le T'$. Therefore there is a a restriction of tableaux in $T'$ to tableaux in $T$ by removing the extra elements. Now we can turn any tableau in $T$ to a tableau into a tableau in the full diagram with lengths $\mbf n_{1}- 1, \dots \mbf n_{k}-k$ by filling each new row $i+t$ with $w_i$.

	\end{proof}


	\subsection{Fixed polynomial case} 
	\label{sub:fixed_polynomial_case}

	This section shows that locally any polynomial can be approximated by a moment curve in such a way that the estimates can be transfered from the moment curve to the polynomial. By a compactness argument, this will allow us to conclude the estimate \eqref{eq:hard_jacobian_estimate} in Lemma \ref{lem:geometric_lemma} for single polynomials, but with a number of open sets that might depend on the polynomial. The use of compactness arguments in this type of set-up is by no means new, and is (for example) already used by Stovall in \cite[Check lemma]{stovall_uniform_2016}

	\begin{lemma}
		[Convergence to the model in the non-degenerate set-up] 
		\label{lem:nondegenerate_continuity}The function $\tilde\Lambda_{\mu}({\bf z})$ is continuous in $({\bf z},\mu)\in \mathbb C^k \times P_n(\mathbb C)^d$, where $P_n(\mathbb C)$ denotes the set of polynomials of degree $n$.
	\end{lemma}

	\begin{proof}
		Consider both the numerator and denominator of $\tilde \Lambda_\mu({\bf z})$ as a polynomial in the components of $\mu$ and $z$. The polynomial $\Lambda_\mu(\bf z)$ on the numerator vanishes on the zero set set $\mathcal Z(v(z_1, \dots, z_k))$, and since this polynomial does not have repeated factors, $v(z_1, \dots, z_k)$ divides the numerator by the Nullstellensatz.
	\end{proof}

	This lemma implies the local version of Theorem \ref{lem:geometric_lemma} around points where the Jacobian does not degenerate:

	\begin{prop}
	\label{prop:local_nondegenerate_jacobian}
		Let $\gamma$ be a polynomial curve in $\mathbb C^d$ in canonical form at $0$, such that $\Lambda^{(d)}(0) \neq 0$. Then there is a neighborhood $B_\epsilon(0)$, with $\epsilon=\epsilon(\gamma)>0$ where \eqref{eq:hard_jacobian_estimate} holds with constant depending only on the dimension.
	\end{prop}

	\begin{proof}
		By the affine invariance of \eqref{eq:hard_jacobian_estimate}, consider a sequence of zoom-ins (see section \ref{sub:prelim} for the definition) in the canonical form parametrized by $\lambda$ that converge to the moment curve (the sequence cannot converge to any other generalized moment curve as the determinant does not vanish at zero). Therefore, it will suffice to show that, for $\lambda$ small enough, the lemma is true for $\mathcal B_\lambda[\gamma]$, that is:

		\begin{equation}\label{eq:nondegenerate_zoomin}
			\left|\frac{J_{\mathcal B_\lambda[\gamma]}(z)}{v(z)}\right|\gtrsim_N \prod_{i=1}^d \Lambda^{(d)}_{\mathcal B_\lambda[\gamma]}(z_i)^{1/d}
		\end{equation}
		
		for $\lambda\le \epsilon$ small enough and $z_i \in B_1(0)$. By re-scaling back (undoing the zoom-in) equation \eqref{eq:nondegenerate_zoomin} will imply the inequality \eqref{eq:hard_jacobian_estimate} for $B_\epsilon(0)$

		For the moment curve (the case $\lambda\to 0$) inequality \eqref{eq:hard_jacobian_estimate} is true, and reads:

		\begin{equation}
			\tilde \Lambda_{\mu}(z_1, \dots, z_d) \gtrsim 1
		\end{equation}
		and since both sides of the inequality converge locally uniformly as $\lambda \to 0$ (the LHS by Lemma \ref{lem:nondegenerate_continuity} and the RHS because it is the $d\nth$ root of a sequence of converging polynomials), the inequality is true for $\lambda$ small enough in the zoom-in.
	\end{proof}

	For the degenerate points where the Jacobian vanishes a similar, but slightly more technical approach gives the same result.

	\begin{lemma}
		[Convergence to the model case in the degenerate set-up with zoom-in] 
		\label{lem:degenerate_convergence}
		Let $W \subseteq C$ be an open sector of small aperture $\epsilon(N,d)>0$ to be determined. Let ${\mbf z}_j$ be a sequence of points in $W^k$ that have norm $\lesssim 1$ and for any $\mbf z_j$ and coordinates $k\neq l$, $(\mbf z_j)_k \neq (\mbf z_j)_l$. Let $\gamma$ be a curve in canonical form at zero, and $\gamma_j:=\mathcal B_{\lambda_j}[\gamma]$, $\lambda_j \to 0$, be a sequence of polynomial curves of degree $N$ that converges to $\mu$, a generalized moment curve of exponents ${\bf n}=(n_1,\dots, n_d)$. Then:
		
		\begin{equation}
			\lim_{j\to \infty} \left(\textbf{}\frac{\Lambda_{\gamma_j}({\bf z}_j)} 
			{|\Lambda_{\mu}({\bf z}_j)|}\right) = 1
		\end{equation}
		and therefore, by the triangle inequality,
		
		\begin{equation}
			\lim_{j\to \infty} \left(\frac{\Lambda_{\gamma_j}({\bf z}_j)} 
			{|\Lambda_{\gamma_j}({\bf z}_j)|}-\frac{\Lambda_{\mu}({\bf z}_j)}
			{|\Lambda_{\mu}({\bf z}_j)|}\right) = 0
		\end{equation}
	\end{lemma}

	\begin{proof}
		First note that it suffices to prove that the lemma is true for a subsequence of the $(\lambda_j,\mbf z_j)$. The claim will follow if we can prove that, for any fixed coordinate $e = e_{l_1} \wedge \dots \wedge e_ {l_k}$ we have:

		\begin{equation}
			\lim_{j\to \infty} \frac{\Lambda_{\gamma_j}({\bf z}_j)|_e}
			{\Lambda_{\mu}({\bf z}_j)|_e} = 1
		\end{equation}
	using the notation $w|_e$ to denote the $e\nth$ co-ordinate of the form $w$. By restricting the problem to the co-ordinates $(e_{l_1}, \dots,  e_ {l_k})$ we may assume $k=d$, and then it suffices to show, in the same set-up of the lemma, that:
	\begin{equation}
			\lim_{j\to \infty} \frac{\tilde \Lambda_{\gamma_j}({\bf z}_j)}
			{\tilde \Lambda_{\mu}({\bf z}_j)} = 1
	\end{equation}

	We will prove this by induction. By passing to a subsequence if necessary, assume without loss of generality that $\mbf z_j$ has a limit. In the base case none of the components of ${\bf z}_j$ has limit zero. In that case, the denominator converges to a non-zero number (since the denominator is a Schur polynomial in the components of ${\bf z}_j$, which does not vanish on a small enough sector by Lemma \ref{lem:schur_positive_complex}) and the result follows. 

	Our first induction case is when all the components go to zero. In this case, by doing a further zoom-in and passing to a further subsequence if necessary, one can reduce to the case where not all the components of ${\bf z}_j$ go to zero. Thus, assume some, but not all the components of ${\bf z}_j$ go to zero.

	Without loss of generality assume it is the first $0<k'<k$ components that go to zero. Let ${\mbf z'}_j:= ((\mbf z_j)_1, \dots,  (\mbf z_j)_{k'})$ be the sequence made by the first $k'$ components of each ${\bf z}_j$, and ${\mbf z}''_j$ the sequence made by the remaining components. Then,

	\begin{equation}
		 \frac{\tilde \Lambda_{\gamma_j}({\bf z}_j)}
			{\tilde \Lambda_{\mu}({\bf z}_j)}
		=  \frac
		{ \sum_{e'\wedge e'' = e}\tilde \Lambda_{\gamma_j}({\bf z}'_j)|_{e'}
			\cdot
		\tilde \Lambda_{\gamma_j}({\bf z}''_j)|_{e''}}
		{ \sum_{e'\wedge e'' = e}\tilde \Lambda_{\mu}({\bf z}'_j)|_{e'}
			\cdot
		\tilde \Lambda_{\mu}({\bf z}''_j)|_{e''}}
	\end{equation}
	we know by the induction hypothesis that each of the terms in the sum in the numerator converges to the corresponding term in the denominator (in the sense that their quotient goes to $1$). So the result will follow if we can prove there is not much cancellation going on on the denominator, that is:

	\begin{equation}
		\limsup_{j \to \infty }
		\frac 
		{ \sum_{e'\wedge e'' = e} |\tilde \Lambda_{\mu}(\mbf z'_j)|_{e'}
		\cdot
		\tilde \Lambda_{\mu}(\mbf z''_j)|_{e''}|} {|\tilde \Lambda_{\mu}({\mbf z}_j)|_e|}<\infty
	\end{equation}
	but this is a consequence of Lemma \ref{lem:transversality_corrected}, because we can bound each of the elements in the sum by 
	$|\tilde \Lambda_{\mu}({\bf z'}_j)|
		\cdot
		|\tilde \Lambda_{\mu}({\bf z''}_j)| \lesssim {|\tilde \Lambda_{\mu}({\bf z}_j)|} $, by equation \eqref{eq:transversality}.

	\end{proof}

	We can also see this lemma in an open-set set-up. Following the same proof as in Proposition \ref{prop:local_nondegenerate_jacobian}, we can prove

	\begin{prop}
	\label{prop:local_degenerate_jacobian}
		Let $\gamma$ be a polynomial curve in $\mathbb C^d$, such that $\Lambda^{(d)}(0) = 0$. Then there is a neighborhood $B_\epsilon(0)$, with $\epsilon=\epsilon(\gamma)$ where \eqref{eq:hard_jacobian_estimate} holds with constant depending only on the dimension.
	\end{prop}

	\begin{proof}
		We have to show that there exists a zoom-in $\mathcal B_{\lambda}[\gamma]$ for $\lambda$ small enough so that the inequality 
		$$
		\left|\frac{J_{\mathcal B_{\lambda}[\gamma]}(z)}{v(z)}\right| \prod_{i=1}^d \Lambda^{(d)}_{\mathcal B_{\lambda}[\gamma]}(z_i)^{- 1/d} \gtrsim_N 1
		$$ holds in the unit ball, but Lemma \ref{lem:degenerate_convergence} implies that

		$$
		\lim_{\lambda_\to 0}
		\left|\frac{J_{\mathcal B_{\lambda}[\gamma]}(z)}{v(z)}\right| \prod_{i=1}^d \Lambda^{(d)}_{\mathcal B_{\lambda}[\gamma]}(z_i)^{- 1/d} = 
		\left|\frac{J_{\mu}(z)}{v(z)}\right| \prod_{i=1}^d \Lambda^{(d)}_{\mu}(z_i)^{- 1/d} \gtrsim 1
		$$
		where the  convergence is locally uniform, and the second inequality is inequality \eqref{eq:hard_jacobian_estimate} for the moment curve.

	\end{proof}

	This finishes the proof that \eqref{eq:hard_jacobian_estimate} if we split compact sets in a finite number of sets that may depend on the polynomial. A small variation of Lemma \ref{prop:local_degenerate_jacobian} can be used at a neighborhood of infinity. In this exposition infinity will be considered simultaneously with the uniform case instead.


	\subsection{Uniformity for polynomials} 
	\label{sub:uniformity_for_polynomials}

	The aim of this section is to show that the number of open sets in the geometric Lemma \ref{lem:geometric_lemma} does not depend on the polynomial. In order to do so, we will show that given a sequence of polynomial curves there exists a subsequence of curves for which \eqref{eq:hard_jacobian_estimate} holds with a uniformly bounded amount of subsets, and thus there must be a uniform bound for all polynomial curves.

	The main challenge in the proof of the uniformity of the number of open sets in Lemma \ref{lem:geometric_lemma} is the case in which the zeros of the Jacobian \textit{merge}, that is, the curves $\gamma_n$ converge to a curve $\gamma$ such that $J_\gamma$ has less zeros than $\gamma$ (without counting multiplicity). We will use zoom-ins near the zeros of $\gamma$ to keep track of this cases. The following lemma is a key tool to do the zoom in:

	\begin{lemma}\label{lem:annuli_continuity}
		Let $\gamma$ be a non-degenerate polynomial curve in $\mathbb C^d$ of degree $N$ such that $$\gamma_i = \prod_{k=1}^{n_j} (z - w_{i,k}) \prod_{l=1}^{m_j} \left(1 - \frac{z}{v_{i,l}}\right)$$, $v_{i,l}, w_{j,k} \in B_{R }\setminus B_{r}$, $n_1<n_2< \dots < n_d$. Then there is a constant $C := C(N, d)$ $ \epsilon:=\epsilon(N,d)>0$ such that \eqref{eq:hard_jacobian_estimate} holds on $W \cap (B_{C^{-1}R }\setminus B_{C r})$ for any sector $W$ of aperture $\le \epsilon$.
	\end{lemma}

	The lemma (and its proof) can be informally stated as: ``If all the zeros of the components of $\gamma$ are far from an annuli, then $\gamma$ behaves like the corresponding moment curve in the annuli". To prove the lemma we will have to re-write it into an equivalent form, more suitable for compactness arguments:

	\begin{lemma} [Lemma \ref{lem:degenerate_convergence}, annuli verison]\label{lem:annuli_convergence}
		Let $\gamma_n\to \mu$ be a sequence of polynomial curves, and $\mu$ a nondegenerate generalized moment curve. Let $w_{(i,k),n}$, $w_{(i,k),n}$ defined as in Lemma \ref{lem:annuli_continuity}, with  $w_{(i,k),n} \to_{n\to \infty} 0$ and $v_{(i,l),n} \to_{n\to \infty} \infty$. Let $r_n$ define a sequence of annuli $A_n = B_0(1)\setminus B_0(r_n)$, so that $\max_{i,k} w_{(i,k),n} = o(r_n)$. Let $\mbf z_n \in (A_n \cap W)^k$, where $W$ is a sector of small enough aperture depending of $n,d$ only. Then:

		\begin{equation}
			\lim_{j\to \infty} \left(\textbf{}\frac{\Lambda_{\gamma_j}({\bf z}_j)} 
			{|\Lambda_{\mu}({\bf z}_j)|}\right) = 1
		\end{equation}
		and therefore, by the triangle inequality,
		
		\begin{equation}
			\lim_{j\to \infty} \left(\frac{\Lambda_{\gamma_j}({\bf z}_j)} 
			{|\Lambda_{\gamma_j}({\bf z}_j)|}-\frac{\Lambda_{\mu}({\bf z}_j)}
			{|\Lambda_{\mu}({\bf z}_j)|}\right) = 0
		\end{equation}
	\end{lemma}

	\begin{proof}
		The proof is the same as the proof in Lemma \ref{lem:degenerate_convergence}. The key difference being the reason why we can zoom in again. In Lemma \ref{lem:degenerate_convergence} the $\gamma_n$ were themselves blow-ups, so blowing up did not change the hypothesis of the lemma. Here, the control of $r_n$ ensures the blow-up will be always at a smaller scale than the scale at which the zeros of $\gamma_n$ are.
	\end{proof}

	\begin{remark}
		Note that in the particular case in which there are no $v_{(i,l),n}$ (that is, all the zeros are going to zero) the annuli can be taken to have exterior radius equal to infinity (that is, the annuli can degenerate to the complement of a disk) or, in the case where all the $w_{(i,k),n}$ are exactly equal to zero, it can be taken to have interior radius equal to 0.
	\end{remark}

	Lemma \ref{lem:annuli_continuity} that we just proved (in its convergence form) lets us control  $J_\gamma$ far from the zeros of the components of $\gamma$. The zeros of the components, however, depend on the co-ordinates we take. In order to solve this, we will show that there we can find a co-ordinate system in which, if we have a zero of a co-ordinate of $\gamma$ that has size $O(1)$ then there is also a zero of $\mathcal J_{\gamma}$ that has size $O(1)$. The following example showcases the relevance of this fact:

    \begin{example}
    Let $\gamma_n = (z, z^3+\delta_n z + \epsilon_n z^2+z^3,z^4+\epsilon z^ 6)$, where $\delta_n, \epsilon_n$ go to zero, but $\epsilon_n$ goes to zero way faster than $\delta_n$. Clearly $\gamma_n \to (z,z^3, z^4)$, and our goal is to find the smallest radius $r_n$ for which we can prove \eqref{eq:hard_jacobian_estimate} outside of $B_{r_n}$. 
    
    The zeros of the first coordinate $z$ are all at zero. The first zero of the second coordinate are at scale $\delta_n^{1/2}$, and the first zero of the third co-ordinate is around ${\epsilon_n}$. Lemma\ref{lem:annuli_continuity} states that it suffices to remove balls of radius $O(\delta_n^{1/2})$. 
    
    The optimal scale at which we should be able to zoom before starting to see the difference between $\gamma_n$ is the scale at which the separation between the zeros of $\Lambda_{\gamma_n}$ is. After a zoom-in at scale $\delta_n^{1/2}$ we obtain the curve $$\tilde \gamma_n(z) = (z, z + \delta_n^{-\frac 1 2} \epsilon_n z^2, z^4+\epsilon_n \delta_n^{-1} z^6).$$ After a change of co-ordinates, we can remove the $t$ monomial on the second co-ordinate, to obtain a new sequence of equivalent curves $$\tilde \gamma(z) = (z, \delta_n^{-\frac 1 2} \epsilon_n z^2+z^3, z^4+\epsilon_n \delta_n^{-1} z^6).$$ This shows that another zoom-in at scale $\delta_n^{-1/2}\epsilon_n$ is actually possible (and therefore that the first zoom-in was in fact not as efficient as possible)

    Lemmas \ref{lem:honest_zeros} and \ref{lem:honest_helper} state that in order to identify the optimal scale at which a curve $\gamma$ stops behaving like a moment curve one must first do a small (close to the identity) change of co-ordinates. The change of co-ordinates is done by eliminating the lower order coefficients that can be eliminated using other co-ordinates. In this case, the second co-ordinate $t^3+\delta_n t + \epsilon_n t^2$ has a term $\delta_n t$ that can be eliminated. A systematic way of doing so (described in Lemma \ref{lem:honest_helper}) is by building a matrix with the coefficients of the polynomials and applying row reductions. In our case the associated matrix would be:
    $$
    \left [
    \begin{matrix}
    1&0&0&0&0&0 \\
    \delta_n & \epsilon_n & 1 &0&0&0\\
    0&0&0&1 &0 & \epsilon_n
    \end{matrix}
    \right]
    $$
    and row-reducing the columns with the indices that lead to the moment curve (in general, one should row-reduce the bold-face positions) brings us to the matrix
    $$
    \left [
    \begin{matrix}
    1&0&\mbf 0&\mbf 0&0&0 \\
    \mbf 0 & \epsilon_n & 1 &\mbf 0&0&0\\
    \mbf 0&0&\mbf 0&1 &0 & \epsilon_n
    \end{matrix}
    \right]
    $$
    or equivalently the curve $\gamma_n = (t, t^3 + \epsilon_n t^2,t^4+\epsilon_n t^ 6)$, that allows for a blow-up at the $O(\epsilon_n)$ scale by Lemma \ref{lem:honest_zeros}.
    \end{example}

	\begin{lemma}[Honest zeros lemma]
	\label{lem:honest_zeros}
		For a non-degenerate polynomial curve $\gamma$, let $R(\gamma)$ be the (absolute value of the) supremum of the zeros of $J_\gamma$ that has absolute value smaller than 1. Let $r_\gamma$ be the supremum of the zeros of the co-ordinates of $\gamma$ (again in absolute value, and counting only the zeros that have absolute value less than 1).

		Then, for any sequence of polynomial curves $\gamma_n \to \mu$, a non-degenerate generalized moment curve, there is a constant $k := k(\gamma_n)$, a sequence of linear operators $L_n \in GL(n;\mathbb C)$ converging to the identity and a sequence of constants $c_n\to 0$ so that:

		\begin{equation}
			R(L_n \gamma_n(z-c_n)) \ge k r(L_n \gamma_n(z-c_n))
		\end{equation}

		In other words, after a suitable change of co-ordinates, controlling the zeros of a sequence of polynomial curves allows us to control the zeros of its Jacobian without significant losses.
 	\end{lemma}

 	\begin{proof}
 		We will assume that $\mu$ is not the standard moment curve, since otherwise the result is trivial because $J_\mu = 1$. We choose the $c_n \to 0$ to re-center the $\gamma_n$ so that $J_{\gamma_n}$ always has a zero at zero. Let $n_i$ be the degree of the $i\nth$ co-ordinate of $\mu$. By composing with suitable $L_n\to Id$ we can assume that the degree $n_i$ component of the $i\nth$ co-ordinate of $\gamma_n$ is always $1$, and the degree $n_j$ of the $i\nth$ component (for $i\neq 0$) is $0$ for all $\gamma_n$. Let $\tilde \gamma_n = L_n \gamma_n(z-c_n)$, then, the following holds:

 		\begin{lemma}
 		\label{lem:honest_helper}
 			Let $\hat\gamma_n$ be a sequence of zoom-ins to $L_n \gamma_n(z-c_n)$ at the scale where the zeros of the co-ordinates appear, and assume $\hat\gamma_n \to \gamma$. Then the multiplicity of the zero of $J_\gamma$ at zero is strictly smaller than the multiplicity of $J_\mu$ at zero.
 		\end{lemma}

 		Using Lemma \ref{lem:honest_helper} above, we can finish the proof by contradiction. Assume $ \gamma_n$ is such that $\tilde \gamma n = L_n \gamma_n(z-c_n)$ contradicts the lemma. Pick a subsequence  for which $R(\tilde\gamma_n)/r(\tilde\gamma_n)$ goes to zero. Let $\hat \gamma_n$ be a zoom-in at the scale at which the first zeros of the components of $\tilde\gamma_n$ appear. Assume, by passing to a subsequence if necesary, that $\hat \gamma_n$ converges. By the hypothesis of $R(\tilde\gamma_n)/r(\tilde\gamma_n)$, it must be that all the zeros of $J:{\hat\gamma_n}$ concentrate back at zero, but this contradicts lemma \ref{lem:honest_helper}.
 	\end{proof}
	

	\begin{proof}
		[Proof (of  Lemma \ref{lem:honest_helper}).]
		Define a matrix $M_{i,k} \in \mathcal M_{(d,N)}(\mathbb C)$ so that $M_{i,k}$ is the coefficient of degree $k$ of the $i$ component of $\gamma$. By the transformation we have done, we know the matrix $M_{i,k}$ has rank $d$, and that for the column $i$, the element $M_{i,n_i}$ is equal to $1$ and for $j>n_i$, $M_{i,n_i}$ is equal to $0$. The multiplicity of $J\mu$ at zero is $\sum n_i - \frac{d^2+d}{2}$. To compute the multiplicity of $J_\gamma$ at zero, we do the following procedure:

		Pick the first column (smallest $k$ index) that is non-zero. Pick the first element of this column (smallest $i$) index that is non-zero. Define $\tilde n_i:=k$, where $i$ is the index that is not zero. Row-reduce $M_{i,k}$ so that the $k\nth$ column is $e_i$. Set all the elements on the right of $(i,n_i)$ to zero. Repeat this process $d$ times. 

		This procedure is a row-reduction and blow-up procedure at the origin, that shows that, at the origin, the polynomial curve looks like a generalized moment curve of degrees $\tilde n_i$, where $\tilde n_i\le n_i$ with at least one of the $\tilde n_i< n_i$. Therefore, the degree of the vanishing of the Jacobian at the origin is strictly smaller.
	\end{proof}
	
	Now we have all the necessary tools to prove \eqref{eq:hard_jacobian_estimate} in  Lemma \ref{lem:geometric_lemma} in the full generality case:

\begin{figure}
    \caption{Proof by picture of estimate \eqref{eq:hard_jacobian_estimate} of Theorem \ref{lem:geometric_lemma}. The strategy of the proof is  by contradiction. By the previous sections, the only thing that remains to prove is that the number of open sets needed to cover a polynomial does not depend on the polynomial curve . Assume $\gamma_n$ is a sequence of polynomial curves of bounded degree for which the necessary number of open sets goes to infinity. The algorithm depicted in the steps above shows how to select a subsequence for which the necessary number of open sets remains bounded, reaching a contradiction.}
    \label{fig:proof by Picture}
    
    \vspace{1em}
    \begin{minipage}{.49\textwidth}
    \centering
    \small
        \textbf{Step 1:} For a suitable convergent subsequence of $\gamma_n$ (after a reparametrization if necessary) isolate all the zeros of $\Lambda^{(k)}_{\gamma_n}$ in a fixed ball $B_0$. Far from that ball, by Lemmas \ref{lem:honest_zeros} and \ref{lem:annuli_continuity} we can split the complex plane into $O(1) $ sectors where \eqref{eq:hard_jacobian_estimate} holds.
        
        \vspace{.2em}
        
        \includegraphics[width = .6\textwidth]{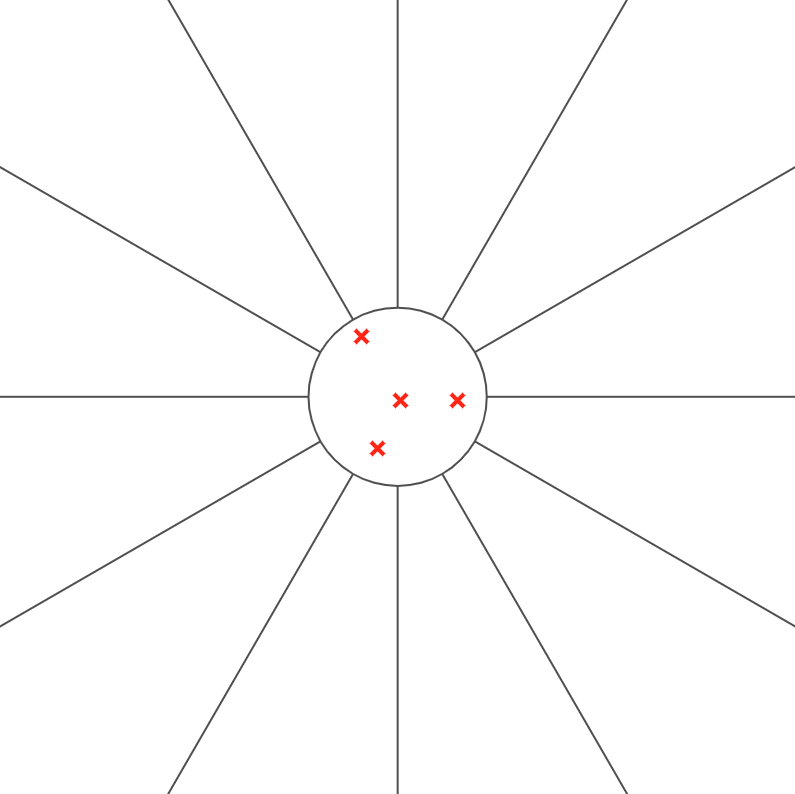}
    \end{minipage}
    \begin{minipage}{.49\textwidth}
    \centering
    \small
    
        \textbf{Step 2:} Since $\gamma_n$ converges, there is a finite number of accumulation points of the zeros of $\Lambda^{(k)}_{\gamma_n}$. By Lemma \ref{prop:local_nondegenerate_jacobian} we can split the complement of a neighbourhood $\bigcup_{i+1}^{O(1)} B_1^{(i)}$ of those accumulation points inside the ball $B_0$.
        
        \vspace{1em}
        
        \includegraphics[width = .6\textwidth]{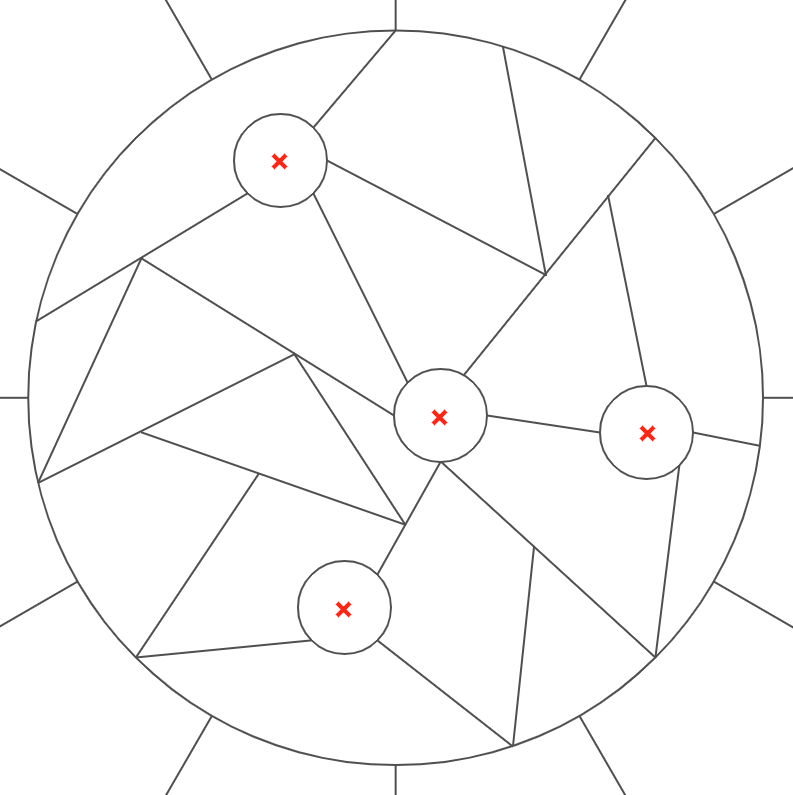}
    \end{minipage}
    
    \begin{minipage}{.49\textwidth}
    \centering
    \small
    
        \textbf{Step 3:} For each of these balls $B_1^{(i)}$, we can apply Lemma \ref{lem:annuli_continuity} again, to split the ball minus a small ball $B_{1,n}^{(i)}$ of radius $r_n$ going to zero into sectors where \eqref{eq:hard_jacobian_estimate} holds independently of $n$. Also, by a small translation, we can assume without loss of generality  that the center of $B_{1,n}^{(i)}$ is a zero of $\Lambda^{(k)}_{\gamma_n}$.
        \vspace{5em}
        
        \includegraphics[width = .6\textwidth]{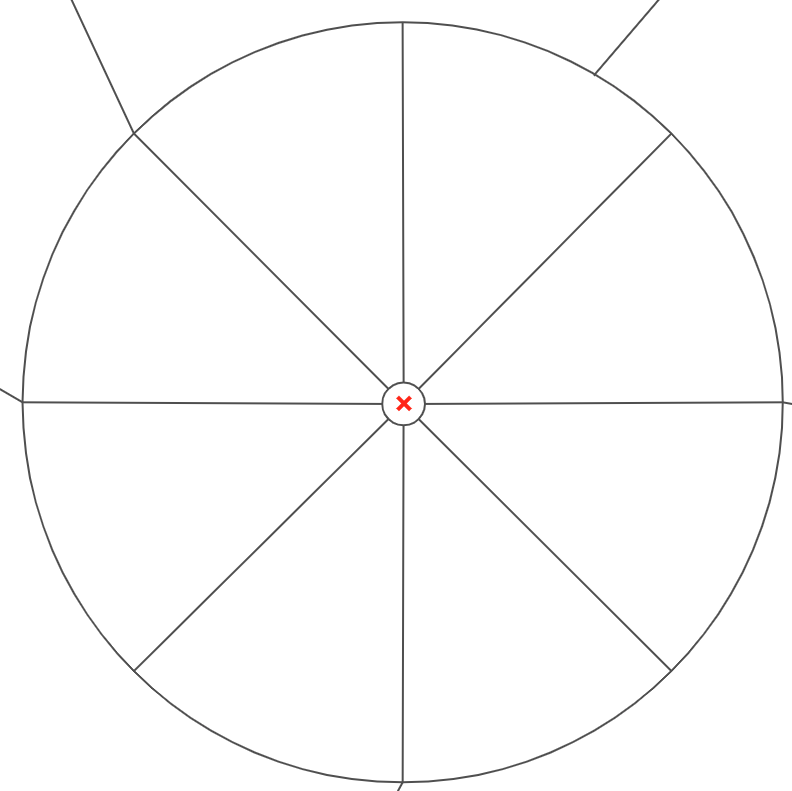}
    \end{minipage}
    \begin{minipage}{.49\textwidth}
    \centering
    \small
    
        \textbf{Step 4:} If all the $r_n$ were zero for a subsequence, by picking that subsequence we would have proven \eqref{eq:hard_jacobian_estimate} for $B_1^{(i)}$ for a given subsequence. Otherwise, zoom into the $r_i$ and pick a convergent subsequence $\tilde \gamma_n$. Now it is possible apply step 2 again. Lemma \ref{lem:honest_zeros} tells us that after zooming there will be zeros of the zoomed in Jacobian do not concentrate at zero, and therefore have lower degree (because the sum of the degrees is conserved), therefore this process will finish in finitely many steps.
    
        \includegraphics[width = .6\textwidth]{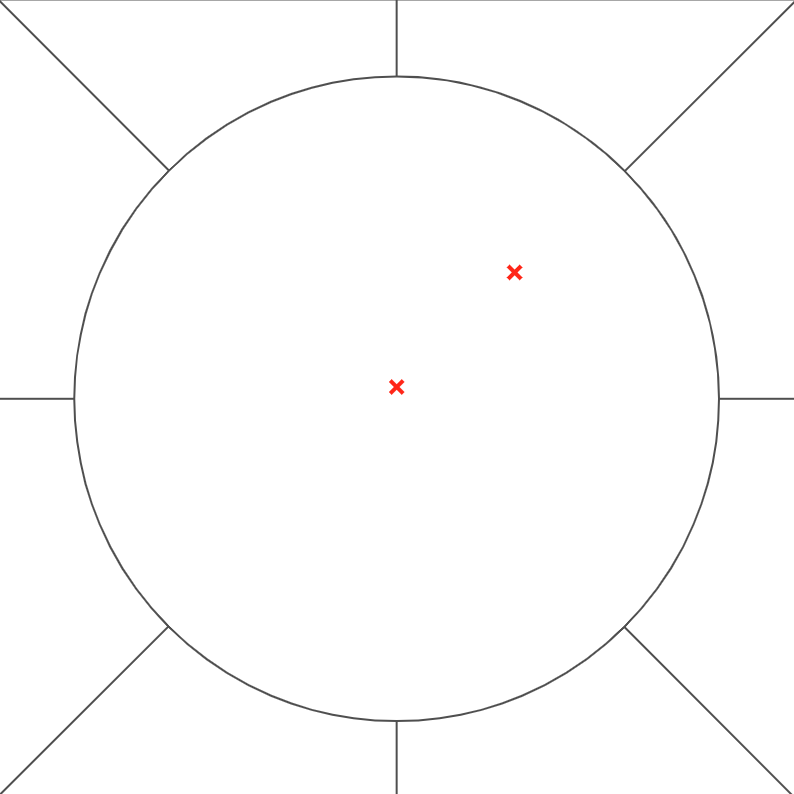}
    \end{minipage}
\end{figure}
	\begin{proof}
{${}^{}$} \\
	\begin{itemize}
		\item The proof will proceed by contradiction. Assume there is a sequence of $\gamma_n$ for which the minimum number of sets need for the geometric Lemma \ref{lem:geometric_lemma} to hold grows to infinity. The contradiction will come from showing that a certain subsequence of the $\gamma_n$ can be covered by a bounded number of subsets.
		\item By passing to a subsequence if necessary, and re-parametrizing, we will assume that the $\gamma_n$ converge to a non-degenerate generalized moment curve $\hat \gamma$, and that all the zeros of $J_{\gamma_n}$ converge to the origin, with one zero of $J_{\gamma_n}$ being exactly at the origin.
		\item By Lemma \ref{lem:annuli_continuity} (Convergence of the Jacobian on annuli with possibly infinite radius), after a suitable reparametrization if necessary, we can cover uniformly $\mathbb C\setminus B_{r_n}$ with sectors so that property \eqref{eq:hard_jacobian_estimate} holds \footnote{Those sectors (up to a small modification that does not change the argument are the generalized triangles defined in Theorem \ref{lem:geometric_lemma})}, where $r_n$ is proportional (with a constant depending on $d,N$) to the size of the biggest zero of $J_{\gamma_n}$. After a suitable change of coordinates by Lemma \ref{lem:honest_zeros}, this is equivalent to consider $r_n$ to be of the size of the biggest zero of a component of $\gamma_n$.
		\item Zoom in to the polynomials $\gamma_n$ at scale $r_n$ to obtain the polynomials $\gamma'_n$. We have to show now that the theorem holds for $\gamma'_n$ on the unit ball. By passing to a subsequence, assume withoutloss of generality that the polynomials converge to a non-degenerate polynomial curve $\gamma'$. Note that the zeros of $J_{\gamma'}$ cannot all converge to the origin (because for each $\gamma'_n$ there is a zero with size $O(1)$). 
		\item By Lemma \ref{lem:annuli_continuity} again, we can find a sequence of annuli of outer radius $O(1)$, centered at the zeros of $\gamma'_n$ so that the condition \eqref{eq:hard_jacobian_estimate} holds after splitting the annuli into sectors.
		\item On the intersection of all the exteriors of all the annuli (by the exterior meaning the connected component containing infinity of the complement of the annuli), property \eqref{eq:hard_jacobian_estimate} holds for $n$ big enough after splitting into $O(1)$ sets, by compactness and Proposition \ref{prop:local_nondegenerate_jacobian} (Local version of \eqref{eq:hard_jacobian_estimate} in the non-degenerate case).
		\item Therefore, it suffices to prove that property \eqref{eq:hard_jacobian_estimate} holds in the interior component of the complement of the annuli. But this can be done by induction: Zoom in into each of those components (which, by hypothesis have lower degree than the original one), and repeat the argument.
	\end{itemize}
	\end{proof}

	\subsection{Injectivity of the $\Sigma$ map} 
	\label{sub:injectivity_of_the_}

	The goal of this section is the last part of Lemma \ref{lem:geometric_lemma}, which we re-state:

	\begin{lemma}
		For each triangle $T_j$ described in the proof of Lemma \ref{lem:geometric_lemma} there is a closed, zero-measure set $R_j \subseteq T_j^d$ so that the sum map $\Sigma(z):=\sum_{i=1}^d \gamma(z_i)$ is $O_N(1)$-to-one in $T_j^d\setminus R_j$.
	\end{lemma}

	\begin{proof}
		The measure zero set $R_j$ is the set of $z_1, \dots , z_d$ where there is $i\neq j$ such that $z_i=z_j$. The fact that $\Lambda_\gamma'(z_1, \dots, z_d)$ does not vanish in $T_j\setminus R_j$ (a consequence of \eqref{eq:hard_jacobian_estimate}) tells us that $(z_1, \dots, z_d)$ does not belong to an irreducible variety of dimension greater than zero of the variety defined by $\{(x_1, \dots, x_d) \in \mathbb C^d | \Lambda_\gamma'(x_1, \dots, x_d) = \Lambda_{\gamma'}(z_1, \dots, z_d)\}$. Therefore, the result follows by Bezout's theorem.
	\end{proof}
	
	\subsection{Reducing back to the real case}
	
	The proof of Theorem \ref{lem:geometric_lemma} for complex curves can actually be extended to real-valued curves. The only thing we need to ensure is that whenever we have to choose a boundary on the partition of $\mathbb C$, the boundary can be chosen to intersect $\mathbb R$ transversely. 
	
	In almost all the estimates above the precise open set to choose is not strictly determined by the algorithm, and one may modify slightly the boundaries so that they intersect $\mathbb C$ transversely. The only exception of \ref{lem:degenerate_convergence}, where we partition $\mathbb C^n$ into sectors that must meet at the origin (or the zero point). The fact that they must meet at the origin only imposes a $1-$dimensional constraint, and one may still choose the wedges ensuring that the boundary of no wedge is contained in $\mathbb R$.
	\newpage
	\section{Uniform restriction for polynomial curves} 
	\label{sec:the_analysis}
	This section outlines the modifications that must be done to the argument \cite{stovall_uniform_2016} to extend it to the complex case. The paper reduces the analytic result (whether the operator is bounded from a certain $L^p$ to a certain $L^q$) to a geometric result, previously proven by Dendrinos and Wright \cite{dendrinos_fourier_2010} in the real case, proven in Section \ref{sec:the_geometry} for the complex scenario.

	Since in this section we will only need to use $L_\gamma:=\Lambda^{(k)}_\gamma$ and $J_\gamma:=\Lambda_\gamma$, we will revert back to the $L_\gamma,J_\gamma$ notation in \cite{stovall_uniform_2016} for convenience. Note that $J_\gamma$ becomes the complex Jacobian (complex determinant of the Jacobian matrix), and that the Jacobian of the associated real transformation will then be $|J_{\gamma}|^2$.

	\subsection{Uniform Local Restriction} 
	\label{sub:uniform_local_restriction}
	
	The first step in the proof is the a local restriction result:

	\begin{thm}[Theorem 2.1 in {[Stovall]}, complex version]\label{thm:uniform_local_restriction}
	\label{lem:small_jacobian_dominates}
	 	Fix $d\ge 2$, $N$, and $(p,q)$ satisfying 
	 	\begin{equation}
	 		p'=\frac{d(d+1)} 2 q, q>\frac{d^2+d+2}{d^2+d}.
	 	\end{equation}
	 	Then, for every ball $B\subseteq \mathbb R^d$ and every degree $N$ polynomial $\gamma: \mathbb C \to \mathbb C^d$ satisfying
	 	\begin{equation}
	 		0<C_1 \le J_\gamma(z) \le C_2, \,\,\, \text{for all} z \in B
	 	\end{equation}
	 	we have the extension estimate
	 	\begin{equation}
	 		\|\mathcal E_\gamma(\chi_B f)\|_q \le C_{d,N,\log {\frac {C_2}{C_1}}} \|f\|_p
	 	\end{equation}
	 \end{thm} 

	An important preliminary fact, proven in \cite[Lemma 2.2]{stovall_uniform_2016} is that whenever $J_\gamma \lesssim 1$ on a open set $\Omega$, there is a change of co-ordinates $L\in SU(d;\mathbb C)$ so that all the coefficients in $L\circ J_\gamma$ are $O_{N,d,\Omega}(1)$. The proof is done through a compactness argument, and transfers without any significant modification to the complex case for convex domains. The exact formulation that we will use is:

	\begin{lemma}
		Fix $N,d \ge 2$, $\epsilon>0$. Then for any polynomial curve $\gamma$ and triangle $T$ with $B_\epsilon \subset T\subset  B(0,1)$  such that $|J_\gamma(T)|\subseteq [1/2, 2]$ there exists a transformation $A \in SU(n;\mathbb C)$ so that $\|A\gamma\|_{C^N(K)} \lesssim_\epsilon 1$.
	\end{lemma}

	\begin{proof}
		Define $\gamma_\epsilon = \epsilon^{\frac{d^2+d}{2}}\gamma(\epsilon^{-1} z)$. Now $\gamma_\epsilon$ has the property that $|J_{ \gamma_\epsilon}(B(0,1))|\subseteq [1/2, 2]$. This reduces the problem to the situation in Lemma 2.2 in \cite{stovall_uniform_2016} (that is, the case where, instead of a triangle $T$ as a domain, we have a ball. The proof therein transfers to the complex case without modifications.
	\end{proof}

	The second preliminary is a statement about offspring curves. Given $\mbf h:= (\mbf h_1, \dots, \mbf h_k)$, the \emph{offspring curve }$\gamma_h(z)$ is defined as $\gamma_h := \frac 1 K \sum_{i=1}^K \gamma(z+h_i)$. They become relevant in order to estimate convolutions. Lemma 2.3 in \cite{stovall_uniform_2016} states:

	\begin{lemma}
	Fix $N,d \ge 2$ and $\epsilon>0$. There exists a constant $c_d > 0$ and a radius $\delta:=\delta(\epsilon,N,d)$ so that for any triangle $B_\epsilon(0) \subseteq T \subseteq B_1$ the following conclusion holds for any polynomial curve $\gamma$ satisfying 

	\begin{equation}
	\label{eq:stovall_offspring_hybound}
		|J_\gamma(T)|\subseteq [1/2, 2]:
	\end{equation}

	for any ball $B$ of radius $\delta$ centered at a point in T, and any $\mbf h:= (\mbf h_1, \dots, \mbf h_k) \in \mathbb C^k$, the curve $\gamma_h$ satisfies the following inequalities on the set $\tilde B_h = \bigcap_{i=1}^k (B - h_j)$:

	\begin{equation}
		\label{eq:stovall_ineq_1}
		|J_{\gamma_h}(\mbf z)| \gtrsim_{N,d,\epsilon} |v( \mbf z)| \prod_{i=1}^d |\Lambda^{(d)}_{\gamma_h}( \mbf z_i)|, \text{ for any }\mbf z \in \tilde B^d_h
	\end{equation}

	\begin{equation}
		\label{eq:stovall_ineq_2}
		|L_{\gamma_h}(z)| \approx_{N,d,\epsilon} 1 \text{ for any } z \in \tilde B_h
	\end{equation}
	in particular, one can cover any such $T$ by $O_{N,d,\epsilon}(1)$ open sets so that the conclusions \eqref{eq:stovall_ineq_1} and \eqref{eq:stovall_ineq_2} hold for any polynomial that follows \eqref{eq:stovall_offspring_hybound}.
	\end{lemma}

	\begin{proof}
		First note, that by choosing $\delta(\epsilon, n,d)$ small enough and Lemma \ref{lem:small_jacobian_dominates}, we have that $$ |J_\gamma(T + B_\delta)|\subseteq [1/4, 4].$$ Now we can cover $T$ by balls of radius $\delta$ and it suffices to prove the Lemma for balls of radius $\delta$. From here on, the result follows as the proof in Lemma 2.3 in \cite{stovall_uniform_2016}, changing the intervals appearing therein for balls.
	\end{proof}

	Using those preliminaries, we can proceed to the proof of Theorem \ref{thm:uniform_local_restriction}. The proof follows exactly as in \cite{stovall_uniform_2016}, which is in itself a sketch of the proof in \cite{drury_restrictions_1985}. The following lemma is a simple computation, that takes the role of the equivalent necessary result in the real case.

	\begin{lemma}
	\label{lem:weak_lp_vandermonde}
		Let $0\le c<\frac 2 d$, then the function $|v(0, z_1, \dots z_{d-1})|^{-2 c}$ belongs to $L^{q}_{B_1^{d-1}(0;\mathbb C)}$ for $1\le q<\frac{2}{d c}$.
	\end{lemma}

	Other than this difference (and the factors of $2$ that appear at the exponent everywhere where the Jacobian appears, that lead to the factor of $2$ in Lemma \ref{lem:weak_lp_vandermonde}), the proof of Lemma 2.4 in \cite{stovall_uniform_2016} applies \textit{mutatis mutandis} to the complex case. We will provide a sketch for completeness:

	\begin{lemma} [Bootsrapping for the extension operator]
		Let $1\le p_0<\frac{d^2+d+2}{2}$, and assume that there is a constant $C_{d,p}$ such that:

		\begin{equation}
			\|\mathcal E_{\gamma_h} (\chi_{I_h}f)\|_{\frac{d(d+1)}2p_0'} \le C_{d,p_0} \|f\|_{L^p_0(\lambda_\gamma)}
			\text{ for all } K>1\text{, }h\in \mathbb C^k
		\end{equation}
		then, for all $p$ satisfying $p^{-1} >\frac{2}{d(d+2)} + \frac{d-2}{d(d+2)} p_0^{-1}$ there exists a $C'(d,p)$ so that

		\begin{equation}
			\|\mathcal E_{\gamma_h} (\chi_{I_h}f)\|_{\frac{d(d+1)}2p_0'} \le C_{d,p_0} \|f\|_{L^{p_0}(\lambda_\gamma)}
			\text{ for all } K>1\text{, }h\in \mathbb C^k
		\end{equation}
	\end{lemma}

\begin{proof}[Sketch of the proof, following \cite{drury_restrictions_1985}]
	The strategy of the proof is to use the convolution-product relationship of the Fourier transform. Proving the theorem for all offspring curves simultaneously will allow us to split the convolution as an integral over offspring curves, with more terms.

	Let $\tilde \gamma := \gamma_h$ be an offspring curve.	
	Since by hypothesis $L_\gamma \approx 1$, we can neglect the factor of $|L_\gamma|^{\frac{4}{d^2+d}}$ in $\lambda_{\gamma}$, and replace $\|f\|_{L^{p_0}(\lambda_\gamma)}$ with the unweighted $\|f\|_{L^{p_0}}$ . Moreover since $J_\gamma(z_1,\dots z_d) \approx v(z_1, \dots z_d) $, we will exchange them freely.

	If we define $g(\xi) := \mu \ast \dots _{d \text{ times}} \dots \ast  \mu (d\cdot \xi)$, a change of variables computation shows:

	\begin{equation}
		g\left(\frac 1 d  \sum_{i=1}^d\tilde \gamma(t_d) \right ) = \frac{c_d}{|J_{\tilde \gamma}(t_1,\dots t_d)|^2} f(t_1) \dots f(t_d).
	\end{equation}
	This motivates the definition (in order to use the offspring curve hypothesis) for $h=(0,h') \in \{0\} \times C^{d-1}$:

	\begin{equation}
		G(t;h):= g\left(\frac 1 d \sum_{i=1}^d\gamma(t+h_i) \right).
	\end{equation}

	We can write now

	\begin{equation}
		\hat g(x) = C_d \int_{\mathbb C^{d-1}} \int_{\tilde B_{(0,h')}} e^{ix \tilde\gamma_h(t)} g(\tilde \gamma_h(t)) |J_{\tilde \gamma}(t+h_1,  \dots t+h_d)|^2 dt dh'
	\end{equation}

	From here, by Plancherel we obtain:

	\begin{equation}
		\|\hat g\|_q \lesssim \|G\|_{L^2_{h'}}(L^q_t;|v(0, h')|^2)
	\end{equation}
	and by the induction hypothesis (and Hölder's inequality - using that $|\tilde B_1|\lesssim 1$) we get:

	\begin{equation}
		\|\hat g \|_q \lesssim \|G\|_{L^1_{h'}(L^p_t;|v(0,h')|^2)}, \,\,\, 1\le p \le p_0, q = \frac{d(d+1)}2 p.
	\end{equation}
	Interpolating between these results, we obtain:

	\begin{equation}
		\|\hat g\|_q \lesssim \|G\|_{L^a_{h'}(L^b_t;|v(0,h')|^2)}
	\end{equation}
	for $(a^{-1},b^{-1})$ in the triangle with vertices $(1/2,1/2), (1,1, (1,p_0^{-1}))$. Therefore it suffices to bound:

	\begin{equation}
		\|G\|_{L^a_{h'}(L^b_t;|v(0,h')|^2)} \sim 
		\left[
			\int d h'
			|v(h)|^{2a-2}
			\left(
			f(t+h_1)
			\right)^{a/b}
		\right]^{1/a}
	\end{equation}
	where we used that $v(t+h_1, \dots t+h_d) = v(h_1,\dots h_d)$ to take the $|v(h)|^{-2}$ term from the inner to the outer integral.  Now, by Lemma \ref{lem:weak_lp_vandermonde}, $|v(h)|^{2(a-1)} \in L^{q}_{h'}$ for $q < \frac{2}{d(a-1)}$ whenever $0< a-1< \frac 2 d$. Interpolating and using Holder, one can see that:

	\begin{equation}
		\|G\|_{L^a_{h'}(L^b_t;|v(0,h')|^2)} \lesssim \|f\|_{L^p_t}^d
	\end{equation}

	whenever $a < b <\frac{2a}{d+2 -da}$, and  $\frac d p < \frac {(d+2)(d-1)} {2} \frac 1 a + \frac 1 b - \frac{d(d-1)}2$.

	Choosing $(a^{-1}, b^{-1})$ arbitrarly close to $\left(\frac{d}{d+2},\frac{d}{d+2}+\frac{1}{p_0} \frac{d+2}{d-2}\right)$ and computing the resulting exponents finishes the proof.

\end{proof}


	\subsection{Almost orthogonality} 
	\label{sub:almost_orthogonality}

	The main result in this section is:

	\begin{lemma}
		Let $\gamma:\mathbb C\to \mathbb C^d$ be a complex polynomial curve, and let $T_j$ be one of the sets in Lemma \ref{lem:geometric_lemma}. For $n\in \mathbb  Z$ define the dyadic partition
		\begin{equation}
			T_{j,n} = \{z \in T_j, |z_j - b_j| \sim 2^n\}
		\end{equation}
		Then for each $(p,q)$ satisfying $q= \frac{d(d+1)}{2}p'$ and $\infty > q > \frac{d^2+d+2}{2}$ and $f \in L^p(d\lambda_\gamma)$ we have:
		\begin{equation}
			\|\mathcal E_\gamma (\chi {T_j} f)\|_{L^q(\mathbb R^{2d})} \lesssim 
			\left\|
			\left(
				\sum_n |\mathcal E_\gamma (\chi {T_{j,n}} f)|^2
			\right)^{1/2}
			\right\|_{L^q(\mathbb R^{2d})}
		\end{equation}
	\end{lemma}

	\begin{proof}
		The proof is a standard Littlewood-Paley argument, using property \eqref{eq:power_geometric_estimates} to localize the support of the Littlewood-Paley blocks, see  \cite[Section 3]{stovall_uniform_2016}.
	\end{proof}

	The goal is now to show that we can sum the pieces in the Littlewood-Paley decomposition above, following  \cite[Lemma 4.1]{stovall_uniform_2016}. The main step to show the summability is to show that terms with different frequencies interact weakly. The following lemma is essentially equivalent to Lemma 4.1 in \cite{stovall_uniform_2016} (more precisely, it is equivalent to equation (4.5) there). Once the lemma is proven, Theorem \ref{thm:restriction} follows by (the proof of) Lemma 4.1 in \cite{stovall_uniform_2016}.
	\begin{lemma}
		There exists $\epsilon(N,d,p)>0$ such that, if $n_1 \le \dots \le n_d$, with $n_d-n_1 > 2d$ and $f_i$ is Schwartz and supported in $T_{j,n_i}$ we have:

		\begin{equation}
			\left \|\prod_{i=1}^d \mathcal E_\gamma [f_i]\right\|_{L^{d+1}} \lesssim 2^{n_D-n_1} \prod_{i=1}^d \|f_i\|_{L^2(d\lambda)}.
		\end{equation}
	\end{lemma}

	\begin{proof}	By Hausdorff-Young,
	\begin{equation}
		\left \|\prod_{i=1}^d \mathcal E_\gamma [f_i]\right\|_{L^{d+1}} \le \| d\mu_1 \ast \dots d\mu_d\|_{L^{\frac{d+1}d}}
	\end{equation}
	where $d\mu_i := \gamma_*(f(t) \lambda_\gamma(t) dt)$, and, if we define $\Phi(\mbf t) := \sum_{i=1}^d \gamma(t_i)$,
	\begin{align*}
		 [d\mu_1 \ast \dots \ast d\mu_d](\phi) 
		 =&
		 \int_{\mathbb C^d} \phi\left( \sum_{i=1}^d \gamma(t_i) \right)\prod_{i=1}^d f_i(t_i )\lambda_\gamma(t_i) d\mbf t
		 \\=&
		 \Phi_*\left[ \prod_{i=1}^df_i(t_i )\lambda_\gamma(t_i) d \mbf t \right ](\phi)
	\end{align*}

	If $\Phi$ was a one-to-one map (as we know it is in the real case) the change of variables rule would give a direct relationship between the $L^p$ norm of the density associated to $d\mu_1 \ast \dots \ast d\mu_d$ and weighted $L^p$ norm of $\prod_{i=1}^df_i(t_i )\lambda_\gamma(t_i)$. Whenever the map is $O(1)$-to-$1$, the following lemma serves as an alternative:

\begin{lemma}
\label{lem:Multivariable_Calc}
	Let $\Omega \subset \mathbb R^k$ be an open set, and $\Phi:\Omega \to \mathbb R^k$ a smooth map, with non-vanishing Jacobian $J_\Phi$ in $\Omega$. Assume that $\Phi$ is at most $n$-to-1 (i.e., each point has at most $n$ pre-images). Let $f:\Omega\to\mathbb R^d$ be a test function, and let $\Phi_\ast f := \frac{d (\Phi_\ast f d \mu_{\text{Leb}(\mathbb R^k)})}{d \mu_{\text{Leb}(\mathbb R^k)}}$ be the measure pushforward of $f$. Then, 

	\begin{equation}
	\label{eq:multivariable_eq1}
		[\Phi_* f](x) = \sum_{\xi: \Phi(\xi)=x} f(\xi) |J_\Phi(\xi)|^{-1}, \text{ and }
	\end{equation}
	\begin{equation}
	\label{eq:multivariable_eq2}
		\| [\Phi_* f](x)\|_{L^p(dx)} \lesssim_n \| f |J_\phi(\xi)|^{1/p-1}\|_{L^p(dy)}
	\end{equation}.
\end{lemma}
\begin{proof}
	The claim \eqref{eq:multivariable_eq1} is a simple calculation in local coordinates, and the second claim \eqref{eq:multivariable_eq2} is a direct consequence of \eqref{eq:multivariable_eq1}.
\end{proof}

This lemma implies the estimate:

\begin{align}
	\left \|\prod_{i=1}^d \mathcal E_\gamma [f_i]\right\|_{L^{d+1}} \lesssim &
	\left  \|
		\prod_{i=1}^d f_i(z_i) \lambda_\gamma(t_i) J_\gamma^{-\frac{2}{d+1}}(\mbf z)
	\right \|_{L^{\frac{d+1}d}(\mbf z)}
	\\ \lesssim &
	\label{eq:variable_change_estimate}
	\left  \|
		\prod_{i=1}^d f_i(z_i) \lambda_\gamma(t_i)^{1/2} v(\mbf z)^{-\frac{2}{d+1}}
	\right \|_{L^{\frac{d+1}d}(\mbf z)}
\end{align}
For $l$ smooth functions $g_1, \dots g_l$ define \footnote{Note that this definition differs from the definition given in \cite{stovall_uniform_2016} by a $\frac{d+1}d$ exponent. The definition given in \cite{stovall_uniform_2016} is $T_{\text{Stovall}}(g_1, \dots g_l) := \|\prod_{i=1}^l f_i(z_i) \lambda_\gamma(t_i)^{1/2} v(z_1, \dots z_l)^{-\frac{2}{d+1}}\|_{L^{{\frac{d+1}d}}(\mathbb C^l)}^{\frac{d+1}d}$.} $$T(g_1, \dots g_l) := \|\prod_{i=1}^l f_i(z_i) \lambda_\gamma(t_i)^{1/2} v(z_1, \dots z_l)^{-\frac{2}{d+1}}\|_{L^{{\frac{d+1}d}}(\mathbb C^l)}.$$ By the pigeonhole principle there is an index $k$ such that $n_{k+1}-n_i \ge \frac{n_d-n_1}d$, and in particular, $n_{k+1}-n_i\ge 2$. In that case:

\begin{equation}
	|v(\mbf z)| \sim \prod_{1\le i<j\le k} |t_i-t_j|  \prod_{1\le i\le k<j} 2^{n_j} \prod_{1\le i\le k<j} |t_i-t_j|
\end{equation}
since all the coupling between variables in \eqref{eq:variable_change_estimate} comes from the Vandermonde determinant, this allows us to split the norm in \eqref{eq:variable_change_estimate} as a product of two norms

\begin{equation}
	T(f_1, \dots , f_d)
	\lesssim
	2^{-\frac{2k}{d+1} \sum_{i>k} n_i} T(f_1, \dots f_k) T(f_{k+1}, \dots f_d).
\end{equation}

In order to control this terms, we use a lemma by Christ. The original statement [CHRIST] is for functions on the real domain, but the argument transfers without change to the complex domain. We provide a sketch of the proof for completeness:

	\begin{lemma}
		\label{lem:Christ's lemma}
	Let $f_i,g_{i,j}$ be test functions, then:
	\begin{equation}
		\int_{\mathbb C^l \sim \mathbb R^{2l}} \prod_{1\le i\le l}f_i(z_i) \prod_{1\le i<j\le l} g_{i,j}(z_i-z_j) dz \lesssim \prod_{i=1}^l \|f_i\|_p \prod_{1\le i<j\le l} \|g_{i,j}\|_{q,\infty}
	\end{equation} 
	whenever $ 2 l =  2 l p^{-1} + \frac{2 l (l-1)}{2} q^{-1}$ and $p^{-1}>l$.

	\end{lemma}
	
	\begin{proof} [Proof of Lemma \ref{lem:Christ's lemma}]

	The proof starts by studying the set of powers $p_i, q_{i,j}$ with $\sum_{i} p_i^{-1} + \sum_{i,j} q_i^{-1} = l$ for which the estimate

	\begin{equation} 
		\int_{\mathbb C^l \sim \mathbb R^{2l}} \prod_{1\le i\le l}f_i(z_i) \prod_{1\le i<j\le l} g_{i,j}(z_i-z_j) dz \lesssim \prod_{i=1}^l \|f_i\|_{p_i} \prod_{1\le i<j\le l} \|g_{i,j}\|_{q_{i,j}}
	\end{equation} 
	holds. We will denote by capital letters the vectors $(p_1^{-1}, \dots p_l^{-1}, q_{1,2}^{-1}, \dots q_{l-1,l}^{-1})$, which we will think of as elements of the affine subspace $H:=\{\sum_{i} p_i^{-1} + \sum_{i,j} q_i^{-1} = l\}$.

	The base cases are $A:=(p_{i}^{-1}=1, q_{i,j}^{-1} = 0)$ and $B:=(p_{i}^{-1}=\delta_{i,1}, q_{i,j}^{-1} = \delta_{i+1,j})$, which follow from Fubini's theorem. Now, the result is invariant over permutations over all the indices $(i,j)$. This allows us to extend the second base case $B$ to all the cases $B_\sigma$ permutations obtained from $B$ by permutations.

	By Riesz-Thorin, the result is then true for $A':=\frac 1 {l!} \sum_{\sigma \in S_l} B_\sigma = (p_i = \frac 1 l, q_i = \frac 1 {2l})$. By Riesz-Thorin again, the result is true for all the points interpolating $A$ and $A'$. This proves the strong version of the theorem.

	To get the weak case, it suffices to show that all the points joining $A$ and $A'$ lie on the interior of the interpolation polytope (interior with the affine topology on $H$). By convexity again, it suffices to show that $A'$ does. The geometric argument can be seen in \cite{christ_restriction_1985} (it is exactly the same as in the real case). 
	\end{proof}

Now by Lemma \ref{lem:Christ's lemma} (see the appendix) and H\"older's inequality (using our control on the size of the supports of $f_i$) we can bound

\begin{equation}
	T(f_1, \dots f_k) \lesssim \prod_{i=1}^k \|f_i\|_{L^{\frac{2d+2}{2d-k+1}}} \lesssim \prod_{i=1}^k 2^{n_i \frac{d-k}{d+1}} \|f_i\|_{2}
\end{equation}
and

\begin{equation}
	T(f_{k+1}, \dots f_d) \lesssim \prod_{i=k+1}^d \|f_i\|_{L^{\frac{2d+2}{d+k+1}}} \lesssim \prod_{i=k+1}^d 2^{n_i \frac{k}{d+1}} \|f_i\|_{2}.
\end{equation}
Joining the estimates above, and using the hypothesis that $n_{k+1}-n_k \ge \frac 1 d (n_d-n_1)$, we get:

\begin{align}
	\left \|\prod_{i=1}^d \mathcal E_\gamma [f_i]\right\|_{L^{d+1}}
		\lesssim&
	2^{\frac{d-k}{d+1} (n_1+ \dots n_k - n_{k+1}- \dots n_d)}\prod_{i=1}^d \|f\|_2 
		\\\lesssim &
	2^{\epsilon_d (n_d-n_1)} \prod_{i=1}^d \|f\|_2 
\end{align}

which finishes the proof.

\end{proof}



	\newpage

	\bibliographystyle{alpha}
	\bibliography{Bibliography}

\end{document}